\documentclass{siamart220329}
\usepackage{cancel}
\usepackage{geometry}
\usepackage[export]{adjustbox}
\usepackage[utf8]{inputenc}
\usepackage{enumitem}
\usepackage{amsmath,amsfonts,amssymb,bm}
\usepackage{array}
\usepackage{float}
\usepackage{graphicx}
\usepackage{overpic}
\usepackage{tikz}
\usepackage{tikz-network}
\usetikzlibrary{shapes}
\usetikzlibrary{arrows}
\usetikzlibrary{graphs}

\usepackage{xcolor}
\usepackage{mathtools}
\newsiamthm{prop}{Proposition}

\newsiamremark{example}{Example}
\newsiamremark{xca}{Exercise}
\newsiamremark{remark}{Remark}
\numberwithin{equation}{section}

\newcommand{\LC}{\mathfrak{LC}}
\newcommand{\Pb}{\mathbb{P}}
\newcommand{\R}{\mathbb{R}}
\newcommand{\N}{\mathbb{N}}
\newcommand{\U}{\mathcal{U}}

\newcommand{\ve}{\varepsilon}
\newcommand{\ep}{\varepsilon}

\newcommand*\circled[1]{\tikz[baseline=(char.base)]{
            \node[shape=circle,draw,inner sep=2pt] (char) {#1};}}

\title{Persistent synchronization of heterogeneous networks with time-dependent linear diffusive coupling}

\author{
Hildeberto Jardón-Kojakhmetov\thanks{
University of Groningen, Faculty of Science and Engineering
Dynamical Systems, Geometry \& Mathematical Physics — Bernoulli Institute, Nijenborgh 9, Groningen, The Netherlands
(\email{h.jardon.kojakhmetov@rug.nl})}
\and
Christian Kuehn\thanks{
Technical University of Munich, School of Computation  Information and Technology, Boltzmannstraße 3, Garching bei München, Germany (\email{ckuehn@ma.tum.de})
}
\and
Iacopo P. Longo\thanks{
Imperial College London, 
Department of Mathematics, DynamIC, 
635 Huxley Building, 180 Queen’s Gate
South Kensington Campus, London, United Kingdom (\email{iacopo.longo@imperial.ac.uk})
}
}

\begin{document}
\maketitle
\begin{abstract}
We study synchronization for linearly coupled  temporal networks of heterogeneous time-dependent nonlinear agents via the convergence of attracting trajectories of each node. The results are obtained by constructing and studying the stability of a suitable linear nonautonomous problem bounding the evolution of the synchronization errors. Both, the case of the entire network and only a cluster, are addressed and the persistence of the obtained synchronization against perturbation is also discussed. Furthermore, a sufficient condition for the existence of attracting trajectories of each node is given. In all cases, the considered dependence on time requires only local integrability, which is a very mild regularity assumption. Moreover, our results mainly depend on the network structure and its properties, and  achieve synchronization up to a constant in finite time. Hence they are quite suitable for applications.  The applicability of the results is showcased via several examples: coupled van-der-Pol/FitzHugh-Nagumo oscillators, weighted/signed opinion dynamics, and coupled Lorenz systems.
\end{abstract}


\section{Introduction}
Within the class of complex systems in nature, technology and society, an interconnected structure is a recurrent feature~\cite{porter2016dynamical,strogatz2001exploring}. Dynamical systems on networks go a long way into successfully capturing and explaining these structures and their evolution~\cite{porter2016dynamical}. For example, the available dynamical theory for static networks---dynamical systems on graphs with a static topology---allows to study problems of opinion formation~\cite{Liggett}, collective motion~\cite{VicsekZafiris}, and epidemic dynamics~\cite{KissMillerSimon}.
In contrast to their static counterpart, temporal networks---sometimes also referred to as time-varying networks, dynamic networks
and temporal graphs---feature a time-dependent variation of the graph structure.  It has become increasingly evident that certain types of natural and societal  interactions require a time-dependent framework~\cite{holme2015modern,holme2012temporal}. Examples include human proximity and communication networks~\cite{saramaki2015seconds, stopczynski2014measuring, zhang2014human}, brain networks~\cite{bassett2013task}, travel and transportation networks~\cite{gallotti2015multilayer}, distributed computing~\cite{kuhn2010distributed}, or ecological and biological networks~\cite{blonder2012temporal} to name just a few. 
However, the analysis of temporal networks is generally much more complicated and the respective mathematical theory remains still substantially open~\cite{ghosh2022synchronized,holme2012temporal}.

For example, the emergence of a synchronized state in a temporal network or in some of its parts---also referred to as clusters---has been investigated analytically and numerically mostly on a case-to-case basis~\cite{boccaletti2006synchronization,jeter2019dynamics,zhang2021designing}; typically under the assumption of switched dynamics---only a finite set of coupling structures are selected over time~\cite{amritkar2006synchronized,cenk2021edges,liberzon1999basic, liu2011synchronization,wen2015pinning}---or of fast switching~\cite{belykh2004blinking,jeter2015synchronization,porfiri2006random,porfiri2008synchronization,rakshit2020intralayer,stilwell2006sufficient}.\par

More general results usually involve the study of the errors between nodes or with respect to a synchronizing trajectory, either through a global Lyapunov function~\cite{belykh2004connection,chen2007synchronization}, a master stability function~\cite{OthmerScriven,SegelLevin,PecoraCarroll,porfiri2011master, MulasKuehnJost}, linearization along a synchronizing trajectory~\cite{lu2005time, lu2008synchronization}, or the Hajnal diameter~\cite{lu2008synchronization}.

In this work, we study synchronization for linearly coupled  temporal networks of $N\ge 2$ heterogeneous time-dependent agents of the form
\begin{equation}\label{eq1}
        \dot x_i = f_i(t,x_i) + \sum_{k=1}^N a_{ik}(t)(x_k-x_i),\quad x_i=x_i(t)\in\R^M,\, i=1,\dots, N,
\end{equation}
where $x_i\in\R^M$ is the state variable of node $i$ and $f_i:\R\times\R^M\to\R^M$ its internal dynamics, and $A:\R\to\R^{N\times N}$, $t\mapsto A(t)=\big(a_{ij}(t)\big)$ is the generalized (or weighted) adjacency matrix of the network. We shall assume that $A$ is a locally integrable function allowing, for example, for discontinuous switching in the network topology. Moreover, we emphasize that, in contrast to the majority of the literature, we allow $a_{ij}$ to take values over the whole real line, which implies that signed networks~\cite{ShiAltafiniBaras} are covered by our theory. Furthermore, the assumptions of regularity in time required for our work are particularly mild: the matrix function $A$ needs to be locally integrable, while the functions $f_i(t,x_i)$ are Lipschitz Carathéodory. In other words, continuity in time is not assumed.\par\smallskip

As for other works on synchronization, also our analysis is based on the study of the evolution of the errors, which, however, are considered in norm $\xi_{ij}(t)=|x_i(t)-x_j(t)|^2$ for $i,j=1,\dots, N$ with $i<j$. This choice allows us to construct a suitable bounding linear system whose stability, in terms of spectral dichotomy, is used to infer information on the finite-time and asymptotic synchronization of the nonlinear problem. The key advantages of our approach can be summarized as follows:\par\smallskip

\begin{itemize}[leftmargin=*]
    \item Only very mild assumptions of regularity akin to local integrability in time are required.
    \item The sufficient conditions for synchronization depend almost completely (and increasingly so in case of global coupling) on the adjacency matrix $A(t)=\big(a_{ij}(t))$. Moreover, such conditions are quantitative and constructive allowing for a control-driven approach to synchronization via the modification of a suitable part of the network.
    \item The results are written for signed networks and for a considerably general class of agents. On the one hand, we admit nodes with internal time-dependent dynamics, a surprisingly rare case in the available theory for temporal networks; synchronization results for nonautonomous agents appear for example in~\cite{caraballo2008synchronization,pereira2013towards} although always with static coupling.  On the other hand, the nodes can be heterogeneous and  synchronization up to a constant is still achieved, provided that the theory applies. In case of homogeneous networks, a result of \emph{exact synchrony} holds. 
    \item The ideas used to obtain synchronyzation can be exploited to investigate its occurrence in either the entire network or only some of its parts (clusters), and we provide sufficient conditions for both these cases. 
    \item Being based on the roughness of the exponential dichotomy, the results of synchrony are robust against perturbation. This fact, which is important by itself, has also striking consequences in the case of perturbation of static networks. For example, we show that strongly connected static networks with positive edges always satisfy our conditions of synchronization provided that a sufficiently large global coupling exists and therefore the achieved synchrony is also robust.
\end{itemize}\par\smallskip

Our work is structured as follows. In Section~\ref{sec:prelimin} the notation, assumptions and some basic notions about nonautonomous linear systems are introduced. Section~\ref{sec:synchronization} deals with the synchronization (up to a constant) of networks of heterogeneous non-autonomous agents which are linearly coupled via a time-dependent network.  
Our approach entails the study of the synchronization errors $\xi_{ij}(t)=|x_i(t)-x_j(t)|^2$ for $i,j=1,\dots, N$ with $i<j$. Firstly we show that the maps $\xi_{ij}(t)$ for $i,j=1,\dots, N$ with $i<j$ induce a dynamical system on the positive cone of $\R^{N(N-1)/2}$. 
Then, we introduce our fundamental assumption~\ref{H1} on the one-sided affine upper-bound to the difference of the vector fields of any pair of uncoupled agents in our network. 
The assumption~\ref{H1} is firstly introduced pointwise and then generalized to pairs of continuous functions in Lemma~\ref{lem:two_nodes_assumptions_from_points_to_functions}, which allows us to upper bound the evolution of any pair of trajectories of two nodes $i$ and $j$ in Lemma~\ref{lem:error_ij_diff_inequality}. 
Theorem~\ref{thm:sync_up_to_constant} is our main result of synchronization which uses the results above to generate a nonautonomous linear inhomogeneous system controlling the evolution of the synchronization errors and therefore synchronization (up to a constant) is achieved via the analysis of stability of this problem: we check sufficient conditions for the existence of an exponential dichotomy with projector the identity. 
Two additional assumptions are made:~\ref{H2} there is a region of the phase space such that the solutions of each node of the network with initial condition therein, remain uniformly ultimately bounded near a globally defined bounded solution;~\ref{H3} the heterogeneity of the nodes can be uniformly bounded on compact intervals of time. Some corollaries address the simpler but important and possibly more common cases that the network is subjected to a global coupling coefficient, and that the assumption \ref{H3} can be changed for a stronger essential boundedness on the whole real line. 
Two examples complete the section to showcase the applicability of the obtained results of synchronization: in the first example we consider a time-dependent network of heterogeneous van der Pol oscillators while in the second one we show how Theorem \ref{thm:sync_up_to_constant} can induce a control-oriented approach to achieve synchronization in a ring network featuring a  contrarian node.

The idea contained in the proof of Theorem~\ref{thm:sync_up_to_constant} are further explored in Section \ref{sec:cluster} to guarantee the synchronization of just a cluster (see Theorem~\ref{thm:sync_cluster}). 
This section is completed by an example of a neural network of heterogeneous FitzHugh-Nagumo where two leading neurons induce recurrent patterns of synchronization on their immediate but possibly shared neighbors.  

In Section \ref{sec:persistence} we briefly expand on the persistence of synchronization achieved via the previous results by using the roughness of the exponential dichotomy. An application to the time-dependent perturbation of synchronized static network is also presented. Incidentally, we also show that strongly connected static networks with positive edges satisfy the conditions of synchronization of our result provided that a sufficiently high global coupling exists. The example of a perturbed star-network of Lorenz systems is presented at the end of the section to showcase the appearance and persistence of synchronization. 

In Section~\ref{sec:suff_cond_attr}, we investigate sufficient conditions for the existence of local attractors for both the uncoupled and the coupled problems. The case of solutions that remain only uniformly ultimately bounded near a bounded reference trajectory is also addressed. In other words, some sufficient results guaranteeing~\ref{H2} are given. 





\section{Notation, assumptions and preliminary definitions}\label{sec:prelimin}
By the symbol $\R^d$, with $d\in\N$, we will denote the $d$-dimensional Euclidean space with  the usual norm $|\cdot|$.
As a rule, a singleton $\{\xi\}$ in $\R^d$ will be identified with the element $\xi$ itself. 
For every $i=1,\dots,d$ the $i$-th component of $\xi\in\R^d$ will be denoted by $\xi_i$. 
Moreover, the notation $\xi\ge0$ means that $\xi_i\ge0$ for all $i=1,\dots,d$, whereas $\xi\gg 0$ means that 
$\xi_i>0$ for every $i=1,\ldots d$. 
The space $\R^d_+$ will denote the set of points $\xi\in\R^d$ such that $\xi\ge0$. This notation naturally extends to vector-valued functions. 
When $d=1$, we will simply write $\R$ instead of $\R^1$ and thus the symbol $\R_+$ will denote the set of non-negative real numbers. 
Moreover, by $B_r$, we denote the closed ball of $\R^d$ centered at the origin.

The symbol $\R^{M\times N}$, with $N,M\in\N$ represents the set of matrices of dimension $M\times N$, and given $A\in\R^{M\times N}$, $A^\top$ will denote its transpose. 
The space $\R^{M\times N}$ will be endowed with the induced norm $\|\cdot\|$ defined by $\|A\|=\sup_{|x|=1}|Ax|$. Oftentimes, we shall write $(\R^M)^N$ to denote the set of $N$-tuples of vectors in $\R^M$. 
Although, this set can evidently be identified with $\R^{M\times N}$, the notation $(\R^M)^N$ will be used when it is more convenient to treat each vector in $\R^M$ separately---typically, when they correspond to the state of a node in a network of size $N$. 
In this case, for every $i=1,\dots,N$ the $i$-th vector of $x\in(\R^M)^N$ will be denoted by $x_i$---the $i$-th node of the network.

For any interval $I\subseteq\R$ and any $W\subset\R^N$, 
$\mathcal{C}(I,W)$ will denote the space of continuous functions from $I$ to $W$ endowed with the sup norm $\|\cdot\|_\infty$. Moreover, $L^\infty$ and $L^1_{loc}$ will denote the spaces of real functions which are respectively, essentially bounded and locally integrable on $\R$, i.e.~belonging to $L^1(I)$ for every bounded interval $I\subset \R$. These latter spaces are intended as each endowed with their usual topology.\par\smallskip

We shall work under the most general assumptions guaranteeing existence and uniqueness of solutions for~\eqref{eq1} (see for example~\cite{bressan2007introduction}). Specifically, the functions  $t\mapsto a_{ij}(t)$ are assumed to be \emph{locally integrable}, whereas the functions $f_i(t,x_i)$ are assumed to be Lipschitz Carathéodory.

\begin{definition} [Lipschitz Carath\'eodory functions]\label{def:LC} 
A function $f\colon\R\times\R^M\to \R^M$  is Lipschitz Carathéodory (in short $f\in\LC$) if it satisfies Carathéodory conditions
\begin{enumerate}[label=\upshape(\textbf{C\arabic*}),leftmargin=27pt,itemsep=2pt]
\item\label{C1} $f$ is Borel measurable and
\item\label{C2} for every compact set $K\subset\R^M$ there exists a real-valued function $m^K\in L^1_{loc}$, called \emph{$m$-bound} in the following, such that for almost every $t\in\R$ one has $|f(t,x)|\le m^K(t)$ for any $x\in K$;
\end{enumerate}
and also a Lipschitz condition
\begin{enumerate}[label=\upshape(\textbf{L}),leftmargin=27pt,itemsep=2pt]
\item\label{L} for every compact set $K\subset\R^M$ there exists a real-valued function $l^K\in L^1_{loc}$ such that for almost every $t\in\R$ one has $|f(t,x_1)-f(t,x_2)|\le l^K(t)|x_1-x_2|$ for any $x_1,x_2\in K$.
\end{enumerate}
\end{definition}

Let us also recall the notion of $L^1_{loc}$-boundedness.
A subset $S$ of positive functions in $L^1_{loc}$  is bounded if for every $r>0$ the following inequality holds
\begin{equation*}
\sup_{m\in S}\int_{-r}^r m(t)\,  dt<\infty\,.
\end{equation*}
In such a case we will say that $S$ is $L^1_{loc}$-bounded.

\begin{remark}
 In many cases in the applications, the functions $m^K$ and $l^K$ appearing in~\ref{C2} and~\ref{L} are in fact constant. This is a particular case of Lipschitz Carathéodory functions where the $m$- and $l$-bounds are in $L^\infty\subset L^1_{loc}$. Consequently, the theory we present still applies. Note also that if a set of functions is uniformly bounded in $L^\infty$, then it is also $L^1_{loc}$-bounded.
 \end{remark}

Furthermore, we will assume that the nodes are subjected to linear diffusive coupling through a time-dependent weighted adjacency matrix $A:\R\to\R^{N\times N}$ so that each entry $a_{ij}(\cdot)\in L^1_{loc}$. More in general, we could consider the problem
\begin{equation}
        \dot x_i = f_i(t,x_i) + \sum_{k=1}^N a_{ik}(\omega_t)(x_k-x_i),\quad x_i\in\R^M,\, \omega\in\Omega,\, i=1,\dots, N,
\end{equation}
where $A:\Omega\to\R^{N\times N}$, $\omega\mapsto A(\omega)=\big(a_{ij}(\omega)\big)$ is the generalized (or weighted) adjacency matrix of the network with respect to a fixed $\omega$ in a complete metric space $\Omega$. A continuous flow $\theta:\R\times \Omega\to\Omega$, $(t,\omega)\mapsto\theta(t,\omega)=\omega_t$ on $\Omega$ can be considered, that, depending on the assumptions, can be completely deterministic or random: if $\Omega=\R$ and $(t,s)\mapsto \theta(t,s)= t+s$, we call the network deterministic; if $(\Omega,\mathcal{F},\Pb)$ is a probability space, we call the network random. 
This formalism would allow the construction of a continuous skew-product flow~\cite{longo2018topologies,longo2017topologies,longo2019weak} and the use of tools from topological dynamics to investigate, for example, the propagation of properties of synchronization. 
Given the complexity of the topic at hand, we prefer to restrict ourselves to a simpler---although less powerful---formalism, to privilege the understanding of the conditions of synchronization, the main focus of our work. In any case, we note that, fixed a certain $\omega\in\Omega$, a path-wise approach through the base flow $\theta$ is also covered by our theory by simply identifying $A(\omega_t)$ with $A(t)$ and dropping the dependence on $\omega$.

As a rule, we shall say that a locally absolutely continuous function $\sigma: I\subset \R\to\R^{M\times N}$, $t\mapsto\sigma(t)=(\sigma_1(t),\dots, \sigma_n(t))$ solves (or is a solution of)~\eqref{eq1} with initial conditions at $t_0\in I$ given by $\overline x=(\overline x_1,\dots,\overline  x_N)\in(\R^M)^N$, if for every  $i=1,\dots,N$, $\sigma_i(\cdot)$ is a solution of the integral problem
\begin{equation}
x_i(t)= \overline x_i + \int_{t_0}^t \bigg(f_i(s,x_i(s)) + \sum_{j=1}^N a_{ij}(s)(\sigma_j(s)-x_i(s))\bigg)\,ds,\quad t\in I.
\end{equation}
We shall denote by $x(\cdot,t_0,\overline x):I\to(\R^{M})^N$ such solution. 

Our objective is to identify conditions of synchronization for~\eqref{eq1}. Let us clarify what we mean by synchronization in this context.

\begin{definition}
We say that~\eqref{eq1} synchronizes up to a constant $\mu>0$ in a synchronization region $E\subset\R\times\R^M$ if all the following properties are satisfied
\begin{itemize}[leftmargin=*,itemsep=2pt]
    \item there is a function $\sigma:\R\to(\R^M)^N$, that is locally absolutely continuous on any compact interval and maps $t\in\R$ to $\sigma(t)=(\sigma_1(t),\dots,\sigma_N(t))$  such that $(t,\sigma_i(t))\in E$ for all $t\in\R$ and all $i=1,\dots,N$, and $\sigma(t)$ solves~\eqref{eq1},
    \item for all $(t_0,\overline x)\in \R\times (E_{t_0})^N$, where $E_{t_0}=\{x\in\R^M\mid (t_0,x)\in E\}$, if the locally absolutely continuous function 
    $x(\cdot, t_0,\overline x):=(x_1(\cdot, t_0,\overline x_1),\dots,x_N(\cdot, t_0,\overline x_N))$ solves~\eqref{eq1} with $x_i(t_0, t_0,\overline x_i)=\overline x_i$ for all $i=1,\dots, N$, then $x(\cdot, t_0,\overline x)$ is defined for all $t>t_0$ and
    \begin{equation}
    \limsup_{t\to\infty}|x_i(t, t_0,\overline x_i)- \sigma_i(t)|\le\frac{\mu}{3},\quad\text{ for all }i=1,\dots,N,
    \end{equation}
    \item for all $i,j=1,\dots,N$
    \begin{equation}
    \limsup_{t\to\infty}|\sigma_i(t)- \sigma_j(t)|\le \frac{\mu}{3}.
    \end{equation}
\end{itemize}
If the previous conditions are satisfied then, the triangular inequality guarantees that any two trajectories of two nodes $i,j$ with initial data  $\overline x_i, \overline x_j\in E_{t_0}$ will satisfy \begin{equation}\limsup_{t\to\infty}|x_i(t, t_0,\overline x_i)- x_j(t, t_0,\overline x_j)|\le \mu.\end{equation}
In particular, if $f_i=f$ for all $i=1,\dots,N$, then we say that~\eqref{eq1} synchronizes in a synchronization region $E\subset\R\times\R^M$ if there is a solution $s(t)$ of $\dot x =f(t,x)$ defined over the whole real line and with graph in $E$ such that for all $(t_0,\overline x)\in \R\times (E_{t_0})^N$, the solution $x(\cdot, t_0,\overline x)$  of~\eqref{eq1}, with $x(t_0, t_0,\overline x)=\overline x$, is defined for all $t>t_0$ and 
\begin{equation}
\lim_{t\to\infty}|x_i(t, t_0,\overline x_i)- s(t)|=0,\quad\text{ for all }i=1,\dots,N.
\end{equation}
\end{definition}

The most important results in our work require a notion of splitting of the extended phase space of a linear time-dependent problem $\dot y =A(t)y$, $y\in\R^d$, in invariant manifolds characterized by asymptotic exponential decay either in forward or in backward time. The notions of exponential dichotomy, dichotomy spectrum and associated splitting fulfill this requirement. We briefly recall them here and point the interested reader to Siegmund~\cite{siegmund2002dichotomy} for all the details of the locally integrable case.  

\begin{definition}[Exponential dichotomy and dichotomy spectrum]\label{def:exp_dichotomy}
Let $ A:\R\to\R^{d\times d}$ be a locally integrable function and consider the linear system 
\begin{equation}\label{eq:linear_non_autonomous}
 \dot y= A(t)y.   
\end{equation}
An invariant projector of~\eqref{eq:linear_non_autonomous} is a function $P:\R\to\R^{d\times d}$
of projections $P(t)$, $t \in\R$, such that
\begin{equation}
P(t)Y(t,s)=Y(t,s)P(s),\quad \text{for all }t,s\in\R,
\end{equation}
where $Y:\R^2\to\R^{d\times d}$ is the principal matrix solution  of~\eqref{eq:linear_non_autonomous} at $s\in\R$,  i.e.~$Y(\cdot,s)$ solves~\eqref{eq:linear_non_autonomous} with $Y(s,s)=Id$.\par\smallskip

The system~\eqref{eq:linear_non_autonomous} is said to have an \emph{exponential dichotomy} on $\R$ if there are an invariant projector $P(\cdot)$ and constants $\gamma>0$, $K\ge1$  such that 
\begin{equation}
\begin{split}
    \big\|Y(t,s)P(s)\big\|\le Ke^{-\gamma(t-s)},\quad\text{for all }t\ge s,&\text{ and}\\
    \big\|Y(t,s)\big(Id-P(s)\big)\big\|\le Ke^{\gamma(t-s)},\quad\text{for all }t\le s,&
\end{split}
\end{equation}
where $Id$ is the identity matrix on $\R^{d\times d}$.
The \emph{dichotomy spectrum} of~\eqref{eq:linear_non_autonomous} is the set
\begin{equation}
\Sigma(A):=\{\alpha\in\R\mid \dot y=(A(t)-\alpha\, Id)y\ \text{ has no exponential dichotomy}\}.
\end{equation}
\end{definition}

\begin{remark}\label{rmk:ED}
(i) Siegmund~\cite{siegmund2002dichotomy} showed that either $\Sigma(A)$ is empty, or it is the whole $\R$, or there exists $k\in\N$, with $1\le k\le d$, such that
\begin{equation}
\Sigma(A)= I_1 \cup[a_2,b_2]\cup\dots \cup [a_{k-1},b_{k-1}]\cup I_k\,,
\end{equation}
where $I_1$ is either $[a_1,b_1]$ or $(-\infty,b_1]$,  $I_k$ is either $[a_k,b_k]$ or $[a_k,\infty)$, and $a_1\le b_1<a_2\le b_2<\dots\le a_{k}\le b_k$.  If $A(\cdot)$ has constant entries, the dichotomy spectrum reduces to the real parts of the eigenvalues of $A$. In addition, a decomposition of $\R\times\R^d$ in spectral manifolds holds, i.e.
\begin{equation}\R\times\R^d=\mathcal{W}_0\oplus\dots\oplus\mathcal{W}_{k+1}\, ;
\end{equation}
see~\cite{siegmund2002dichotomy} for details. If the dichotomy spectrum of~\eqref{eq:linear_non_autonomous} is contained in $(-\infty,0)$, then~\eqref{eq:linear_non_autonomous} admits an exponential dichotomy with invariant projector $P(\cdot)\equiv Id$.\par\smallskip
(ii) The notion of exponential dichotomy should be regarded as a natural extension  to the nonautonomous framework of the tools for stability analysis of autonomous linear systems. 
Indeed, it is well-known that the time-dependent eigenvalues are of little help for the investigation of stability of nonautomous linear problems~\cite{coppel2006dichotomies,fink1974almost}, unless the time-dependence is periodic (replace eigenvalues for Floquet multipliers~\cite{hale1969ordinary}) or slow~\cite{coppel2006dichotomies,potzsche2004exponential}. On the other hand, Lyapunov exponents successfully encapsulate the asymptotic stability but do not necessarily provide robustness against nonlinear perturbations, which is instead guaranteed by the roughnees of the exponential dichotomy~\cite{coppel2006dichotomies}---the Lyapunov spectrum is a subset of the dichotomy spectrum and if the latter is a point spectrum, then it reduces to the former~\cite{sacker1978spectral,dieci2002lyapunov}. The downside of the notion of exponential dichotomy resides in the difficulty of verifying it. Some explicit calculable sufficient criteria were provided by Coppel in terms of the time-dependent eigenvalues for bounded matrix functions~\cite{coppel2006dichotomies} and by Fink in terms of row- or column-dominance of the matrix $A(t)$~\cite{fink1974almost}, which we will use later in this work. Efficient numerical approximations have also been developed using QR and singular value decomposition~\cite{dieci2002lyapunov2,dieci2002lyapunov,dieci2007lyapunov,froyland2013computing}.
\end{remark}

\section{Synchronization}\label{sec:synchronization}

In this section, we provide sufficient conditions for synchronization up to a constant via comparison arguments with a suitably constructed linear system bounding the growth of the ``errors". Given a solution $x:\R\to(\R^M)^N$ of~\eqref{eq1}, we shall call \emph{(synchronization) errors}, the vector $\xi(t)\in\R^{N(N-1)/2}$ defined by 
\begin{equation}
\xi_{ij}(t)=|x_i(t)-x_j(t)|^2,\quad i,j=1,\dots,N,\, i<j.
\end{equation}

Our first observation is that given a network~\eqref{eq1}, the synchronization errors induce a nonautonomous dynamical system on the positive cone $\R^{N(N-1)/2}_+$ of $\R^{N(N-1)/2}$.
\begin{prop}\label{prop:dyn_sys_from_difference}
Given a network of $N$ nodes, $N\ge 2$, in $\R^M$ as in~\eqref{eq1}, the map 
\begin{equation}\label{eq:08/11-14:24}
x\in(\R^M)^N\mapsto (|x_i-x_j|^2)_{i,j=1,\dots,N,\,i<j}\in\R^{N(N-1)/2}
\end{equation}
induces a nonautonomous dynamical system on $\R^{N(N-1)/2}_+$.
\end{prop}
\begin{proof}
Let $(t_0,x_0)\in\R\times(\R^M)^N$ and consider the solution $x(\cdot,t_0,x_0):\R\to(\R^M)^N$ of~\eqref{eq1} with initial data $x(t_0)=x_0$. Note that, by definition of solution, the cocycle property $x(t+s,t_0,x_0)=x\big(t,s,x(s,t_0,x_0)\big)$ is satisfied whenever all the involved terms are well-defined. Consequently, a local cocycle is induced also on $\R^{N(N-1)/2}$ through the map in~\eqref{eq:08/11-14:24}. In particular, note that the obtained trajectories verify
\begin{equation}
\begin{split}
\frac{d}{dt}|x_i&(t,t_0,x_0)-x_j(t,t_0,x_0)|^2\\
&=2\langle x_i(t,t_0,x_0)-x_j(t,t_0,x_0), f_i\big(t,x_i(t,t_0,x_0)\big)-f_j\big(t,x_j(t,t_0,x_0)\big)\rangle  \\
&\qquad+\langle x_i(t,t_0,x_0)-x_j(t,t_0,x_0), \sum_{k=1}^N a_{ik}(t)\big(x_k(t,t_0,x_0)-x_i(t,t_0,x_0)\big)\rangle \\
&\qquad-\langle x_i(t,t_0,x_0)-x_j(t,t_0,x_0), \sum_{k=1}^N a_{jk}(t)\big(x_k(t,t_0,x_0)-x_j(t,t_0,x_0)\big)\rangle 
\end{split}
\end{equation}
for $i,j=1,\dots,N,\,i<j$. Therefore, we have that if $|x_i(t,t_0,x_0)-x_j(t,t_0,x_0)|^2=0$ for some $i,j=1,\dots,N,\,i<j$ and $t\in\R$, then $\frac{d}{dt}|x_i(t,t_0,x_0)-x_j(t,t_0,x_0)|^2=0$. Hence, the positive cone $\R^{N(N-1)/2}_+$ is invariant for the induced flow.
\end{proof}

The following pair-wise property on the decoupled dynamics will be important for our argument.

\begin{enumerate}[label=\upshape(\textbf{H1}),leftmargin=27pt,itemsep=2pt]
\item\label{H1} For every $i,j=1,\dots,N$ and $r>0$ there are functions $\alpha^r_{ij},\beta^r_{ij}\in L^1_{loc}(\R,\R)$, with $\beta^r_{ij}$ non-negative, such that for almost every $t\in\R$
\begin{equation}\label{eq:H1}
\langle x-y,f_i(t,x)-f_j(t,y)\rangle \le \alpha^r_{ij}(t)\,|x-y|^2+\beta^r_{ij}(t),\quad\text{for all }x,y\in B_r.
\end{equation}
\end{enumerate}

\begin{remark}\label{rmk:f_i=f_j implies_beta=0}
Hypothesis~\ref{H1} seems somewhat technical but it is in fact not very restrictive. Note that if the network is made of identical nodes, i.e.~if $f_i=:f$ for all $i=1,\dots,N$, then~\ref{H1} is a one-sided Lipschitz condition. In particular, it is trivially true with $\alpha^r_{ij}(t)=l^r(t)$ being the Lipschitz coefficient of $f$ on $B_r$ and $\beta_{ij}^r(t)=0$ for all $r>0$ and $t\in\R$. Indeed, using the Cauchy-Schwarz inequality and the Lipschitz continuity, we have that
\begin{equation}
\langle x-y,f(t,x)-f(t,y)\rangle \le |x-y|\cdot|f(t,x)-f(t,y)|\le l^r(t)|x-y|^2.
\end{equation}
Note also that, in practical cases, one can often choose the functions $\alpha^r_{ij},\beta^r_{ij}\in L^1_{loc}(\R,\R)$ as constants. It goes without saying that the theory hereby developed still applies. In fact, this simplification allows sharper results as we highlight in some of the corollaries to our main theorems. 

In any case, it is helpful to have a rough intuition of what $\alpha_{ij}$ and $\beta_{ij}$ in Hypothesis~\ref{H1} represent. The function $\beta_{ij}^r(t)$ can be interpreted as a measure of how distinct the dynamics of nodes $i$ and $j$ is; after all, $\beta_{ij}^r(t)\geq0$ vanishes when $f_i=f_j$. On the other hand $\alpha_{ij}^r(t)$ tells us the ``tendency of the decoupled solutions to synchronize''. In the homogeneous scalar case, if we call $e(t)=x(t)-y(t)$, we have that hypothesis~\ref{H1} would read as $e(t)\dot e(t)\leq\alpha_{ij}^r(t)e(t)^2$. If $\alpha_{ij}^r(t)<0$ for all $t\in\R$, indeed the ``decoupled error'' $e(t)$ vanishes as $t\to\infty$. This interpretation becomes also evident when we present our main result, Theorem~\ref{thm:sync_up_to_constant}. If the nodes are identical, then  $\alpha_{ij}^r(t)=\alpha^r(t)<0$ for all $t\in\R$ implies the dissipativity of the uncoupled problem $\dot x=f(t,x)$ and therefore the existence of a forward attracting trajectory.
\end{remark}

Next, we show that it is possible to extend~\ref{H1} to pairs of continuous functions obtaining the same inequality almost everywhere.

\begin{lemma}\label{lem:two_nodes_assumptions_from_points_to_functions}
Assume~\ref{H1} holds and consider two continuous functions $\phi,\psi:I\to B_r$, with $I\subseteq \R$ and $r>0$.Then, for almost every~$t\in I$,
\begin{equation}\label{eq:one_sided_Lips_TIME}
\big\langle \phi(t)-\psi(t),f_i\big(t,\phi(t)\big)-f_j\big(t,\psi(t)\big)\big\rangle \le \alpha^r_{ij}(t)\,|\phi(t)-\psi(t)|^2+\beta^r_{ij}(t).
\end{equation}
\end{lemma}
\begin{proof}
The proof of this statement is based on the one of Lemma 3.3 in~\cite{longo2021monotone}. Consider two functions $\phi,\psi:I\to B_r$ and let $D=\{s_n \mid n\in\N\}$ be a dense subset of $I$. From~\ref{H1} we know that given $n\in\N$ there is a subset $J_n\subset I$ of full
 measure such that for every $t\in J_n$,
\begin{equation}\label{Jn}
 \big\langle \phi(s_n)-\psi(s_n),f_i\big(t,\phi(s_n)\big)-f_j\big(t,\psi(s_n)\big)\big\rangle \le \alpha^r_{ij}(t)\,|\phi(s_n)-\psi(s_n)|^2+\beta^r_{ij}(t).
\end{equation}
Next we consider the subset $J=\bigcap_{n=1}^\infty J_n\subset I$, also of full measure and fix $t\in J$. From the density of $D$ we find a subsequence
$(s_{n_k})_{k\in N}$ such that $\lim_{k\to\infty} s_{n_k}=t$, and
from~\eqref{Jn} we deduce that for each $k\in\N$
\begin{equation}\label{ineqsn}
\begin{split}
\big\langle \phi(s_{n_k})-\psi(s_{n_k}),f_i\big(t,\phi(s_{n_k})\big)&-f_j\big(t,\psi(s_{n_k})\big)\big\rangle \\
&\le \alpha^r_{ij}(t)\,|\phi(s_{n_k})-\psi(s_{n_k})|^2+\beta^r_{ij}(t).
\end{split}
\end{equation}
Moreover, from~\ref{L} and the continuity of the functions $\phi$ and $\psi$ we obtain
\begin{align*}
\lim_{k\to\infty}f_i\big(t,\phi(s_{n_k})\big)&=f_i\big(t,\phi(t)\big),\quad\text{and}\quad
\lim_{k\to\infty}f_j\big(t,\psi(s_{n_k})\big)=f_j\big(t,\psi(t)\big),
\end{align*}
which together with~\eqref{ineqsn} yields~\eqref{eq:one_sided_Lips_TIME} for $t\in J$, and finishes the proof.
\end{proof}

Now, we proceed to proving a differential inequality that will be the fundamental block to subsequently construct our comparison argument.

\begin{lemma}\label{lem:error_ij_diff_inequality}
Assume~\ref{H1} holds and that for every $i=1,\dots,N$ there is a bounded absolutely continuous function $\sigma_i:\R\to\R^M$, such that $\sigma(t)=(\sigma_1(t),\dots,\sigma_N(t))$ solves~\eqref{eq1}.
Then, for every $i,j=1,\dots,N$, the following inequality holds,
\begin{equation}
\begin{split}
\frac{1}{2}\frac{d}{dt}&\,|\sigma_i(t)-\sigma_j(t)|^2\le\delta_{ij}(t)\,|\sigma_i(t)-\sigma_j(t)|^2+\beta^\rho_{ij}(t)+\\[-2pt]
&\qquad\quad+\frac{1}{2}\sum_{\substack{k=1\\k\neq i,j}}^N\big(a_{jk}(t)-a_{ik}(t)\big)\,\big(|\sigma_i(t)-\sigma_k(t)|^2-|\sigma_j(t)-\sigma_k(t)|^2\big)
\end{split}
\end{equation}
where 
\begin{equation}
\delta_{ij}(t)=\alpha^\rho_{ij}(t)-\Big(a_{ij}(t)+a_{ji}(t)
    +\frac{1}{2}\sum_{\substack{k=1\\k\neq i,j}}^N\big(a_{ik}(t)+a_{jk}(t)\big)\Big).
\end{equation}
\end{lemma}
\begin{proof}
By definition, we have that 
\begin{equation}
\begin{split}
    \frac{1}{2}\frac{d}{dt}|\sigma_i(t)-\sigma_j(t)|^2&=\big\langle\sigma_i(t)-\sigma_j(t),f_i\big(t,\sigma_i(t)\big)-f_j\big(t,\sigma_j(t)\big)\big\rangle+\\
    &\qquad\qquad+\big\langle\sigma_i(t)-\sigma_j(t),\sum_{k=1}^N a_{ik}(t)\big(\sigma_k(t)-\sigma_i(t)\big)\big\rangle+\\
    &\qquad\qquad-\big\langle\sigma_i(t)-\sigma_j(t),\sum_{k=1}^N a_{jk}(t)\big(\sigma_k(t)-\sigma_j(t)\big)\big\rangle.
\end{split} 
\end{equation}
From Lemma~\ref{eq:one_sided_Lips_TIME} and the boundedness of $\sigma(t)$, we immediately have that for almost every $t\in\R$,
\begin{equation}
\big\langle\sigma_i(t)-\sigma_j(t),f_i\big(t,\sigma_i(t)\big)-f_j\big(t,\sigma_j(t)\big)\big\rangle\le \alpha^\rho_{ij}(t)\,|\sigma_i(t)-\sigma_j(t)|^2+\beta^\rho_{ij}(t),
\end{equation}
for some $\rho>0$. On the other hand, recalling that $2\langle x,y\rangle=|x|^2+|y|^2-|y-x|^2$, the following chain of inequalities holds true
\begin{equation}
\begin{split}    &\big\langle\sigma_i(t)-\sigma_j(t),\,
    \sum_{k=1}^N\big[ a_{ik}(t)\big(\sigma_k(t)-\sigma_i(t)\big)- a_{jk}(t)\big(\sigma_k(t)-\sigma_j(t)\big)\big]\big\rangle\\
    &\qquad=-\big(a_{ij}(t)+a_{ji}(t)\big)\,|\sigma_i(t)-\sigma_j(t)|^2-\sum_{\substack{k=1\\k\neq j}}^Na_{ik}(t)\big\langle\sigma_i(t)-\sigma_j(t),\,
    \sigma_i(t)-\sigma_k(t)\big\rangle\\
    &\qquad\qquad\quad-\sum_{\substack{k=1\\k\neq i}}^Na_{jk}(t)\big\langle\sigma_i(t)-\sigma_j(t),\,
    \sigma_k(t)-\sigma_j(t)\big\rangle\\
    &\qquad=-\big(a_{ij}(t)+a_{ji}(t)\big)\,|\sigma_i(t)-\sigma_j(t)|^2+\\
    &\qquad\qquad\quad-\,\sum_{\substack{k=1\\k\neq j}}^N\frac{a_{ik}(t)}{2}\big(|\sigma_i(t)-\sigma_j(t)|^2+|\sigma_i(t)-\sigma_k(t)|^2-|\sigma_k(t)-\sigma_j(t)|^2\big)+\\
    &\qquad\qquad\quad-\sum_{\substack{k=1\\k\neq i}}^N\frac{a_{jk}(t)}{2}\,\big(|\sigma_i(t)-\sigma_j(t)|^2+|\sigma_k(t)-\sigma_j(t)|^2-|\sigma_i(t)-\sigma_k(t)|^2\big)\\
    &\qquad=-\Big(a_{ij}(t)+a_{ji}(t)+\frac{1}{2}\sum_{\substack{k=1\\k\neq i}}^N\big(a_{ik}(t)+a_{jk}(t)\big)\Big)\,|\sigma_i(t)-\sigma_j(t)|^2\\
    &\qquad\qquad\quad+\frac{1}{2}\sum_{\substack{k=1\\k\neq i,j}}^N\big(a_{jk}(t)-a_{ik}(t)\big)\,\big(|\sigma_i(t)-\sigma_k(t)|^2-|\sigma_j(t)-\sigma_k(t)|^2\big)
\end{split}
\end{equation}
Gathering all the previous formulas we obtain the sought-for inequality.
\end{proof}

In order to discuss the question of synchronization for systems like~\eqref{eq1}, we still need two more properties besides~\ref{H1}. First of all, we shall assume that for every $i=1,\dots,N$, there is a reference trajectory $\sigma_i:\R\to\R^M$, and a tubular  neighborhood around it, towards which the dynamics of the node $i$ converges as time increases. If the trajectories $\sigma_i(\cdot)$ are in fact (locally) attractive, then we shall talk of synchronization of attracting trajectories. In order to rigorously present our assumption, let us recall the notion of nonautonomous set.

\begin{definition}
    A nonautonomous set is a subset of the extended phase space $\U\subset\R\times\R^M$. A $t$-fiber of $\U$ is defined as $\U_t=\{x\in\R^M\mid (t,x)\in\U\}$. A nonautonomous set is called forward invariant if  $x(t,t_0,\U_{t_0})\subseteq \U_t$ for all $t>t_0$. In general, $\U$ is said to have a topological
property (such as compactness or closedness) if each fiber of $\U$ has this property. We shall say that this property is uniform if it is uniformly true for every fiber.
\end{definition}
Our assumption on local uniform ultimate boundedness of solutions reads as, 
\begin{enumerate}[label=\upshape(\textbf{H2}),leftmargin=27pt,itemsep=2pt]
\item\label{H2} 
consider~\eqref{eq1} and assume that there are $\U\in\R\times\R^M$, $\mu\ge0$, and for every $i=1,\dots,N$ there is a bounded locally absolutely continuous function $\sigma_i:\R\to\R^M$ such that $\sigma(t)=(\sigma_1(t),\dots,\sigma_N(t))^\top\in\U_t$ solves~\eqref{eq1}, $\sup_{i=1,\dots, N}\|\sigma_i(\cdot)\|_\infty<\rho$ for some $\rho>0$, and if additionally $x(\cdot,t_0,\overline x)=(x_1(\cdot,t_0,\overline x_1),\dots, x_N(\cdot,t_0,\overline x_N))^\top$ solves~\eqref{eq1} with initial conditions $x_i(t_0)=\overline x_i\in\U_{t_0}$, then $x(\cdot,t_0,\overline x)$ is defined for all $t\ge t_0$ and 
\begin{equation}\label{eq:uniform_ultimate_boundedness}
\limsup_{t\to\infty} |x_i(t,t_0,\overline x_i)-\sigma_i(t)|\le\frac{\mu}{3},\quad\text{for all }i=1,\dots,N.
\end{equation}
\end{enumerate}



In Section~\ref{sec:suff_cond_attr}, we provide sufficient conditions for~\ref{H2} to hold true. The assumption \ref{H2} can be alternatively read as the existence of an inflowing invariant open ball of diameter $\mu>0$ for the networked system. Indeed, if this is true a globally defined bounded solution of \eqref{eq1} can be obtained as pullback limit of any initial condition on the boundary of such ball.

Furthermore, we shall assume that heterogeneity of the nodes can be uniformly bounded on compact intervals of time, i.e.~
\begin{enumerate}[label=\upshape(\textbf{H3}),leftmargin=27pt,itemsep=2pt]
\item\label{H3} the set $\{\beta_{ij}^\rho(\cdot+t)\in L^1_{loc}\mid t\in\R,\, i,j=1,\dots,N\}$ is $L^1_{loc}$-bounded, and call 
\begin{equation}
0\le\mu_1:=\sup_{\tau\in\R}\int_\tau^{\tau+1}|\beta(s)|\, ds<\infty,
\end{equation}
where $\beta(t)=(2\beta^\rho_{ij}(t))_{i,j=1,\dots,N,\,i<j}^\top$.
\end{enumerate}




\subsection{Synchronization up to a constant of the entire network}
The following theorem is our main result of synchronization up to a constant for heterogeneous time-dependent linearly coupled networks. Note that, due to assumption~\ref{H2}, any function $x(t)=(x_1(t),\dots,x_N(t))$ which solves~\eqref{eq1} with initial data in $\U$ also satisfies~\eqref{eq:uniform_ultimate_boundedness}. 
Hence, in order to investigate the synchronization of the system~\eqref{eq1} for initial conditions of the nodes in $\U$, it is sufficient to study the asymptotic behavior of $|\sigma_i(t)-\sigma_j(t)|$ for all $i,j=1,\dots,N$. 

\begin{theorem}\label{thm:sync_up_to_constant}
Assume~\ref{H1},~\ref{H2} and~\ref{H3} hold and fix any $M>\mu_1$. 
If there is  $t_0\in\R$ such that for all $i,j=1,\dots,N$, with $i<j$, and for almost every $t>t_0$,
\begin{equation}\label{eq:cond_sync_1}
\begin{split}
\delta_{ij}(t):= \alpha^\rho_{ij}(t)-\Big(a_{ij}(t)+a_{ji}(t)
    +\frac{1}{2}\sum_{\substack{k=1\\k\neq i,j}}^N\big(a_{jk}(t)+a_{ik}(t)\big)\Big)<0\\
\end{split}    
\end{equation}
and additionally 
\begin{equation}\label{eq:cond_sync_2}
\overline \gamma:=\inf_{t\in\R}\min_{\substack{i,j=1,\dots,N,\\ i\neq j}}\Big\{\underbrace{2|\delta_{ij}(t)|-\sum_{\substack{k=1\\k\neq i,j}}^N\big|a_{jk}(t)-a_{ik}(t)\big|}_{=:\gamma_{ij}(t)}\Big\}>-\log\left(1-\frac{\mu_1}{M}\right),
\end{equation}
then, for every $\ep>0$ there is $T(\ep)=\frac{1}{\gamma}\ln\left(4\rho^2/\ep\right)>0$ such that for all $i,j=1,\dots,N$
\begin{equation}
|\sigma_i(t)-\sigma_j(t)|^2< \ep+M,\qquad\text{for } t-t_0>T(\ep).
\end{equation}
In other words,~\eqref{eq1} synchronizes up to a constant in finite time.

\end{theorem}

\begin{proof}
Firstly, let us  introduce the continuous function $\eta:\R\to\R_+$ defined by 
\begin{equation}
\eta(a)=
\begin{cases}
0&\text{if } a\le0,\\
a&\text{if } a>0,
\end{cases}
\end{equation}
which will be used multiple times within this proof. 
Consider the synchronization errors given by the vector of $N(N-1)/2$ components 
\begin{equation}\xi(t)=\big(\xi_{ij}(t)\big)_{i,j=1,\dots,N,\,i<j}^\top,\qquad\text{where}\qquad\xi_{ij}(t)=|\sigma_i(t)-\sigma_j(t)|^2.
\end{equation}
Thanks to Lemma~\ref{lem:error_ij_diff_inequality}, we have that $\dot\xi(t)$ is a \emph{sub-function} with respect to the initial value problem $u'= E(t)u+\beta(t)$, $u(t_0)=\xi(t_0)$, i.e.,
\begin{equation}\label{eq:control}
\dot \xi(t)\le E(t)\xi(t)+\beta(t),\quad\text{for all }t>t_0,\,t_0\in\R,
\end{equation}
where 
$\beta(t)=(2\beta^\rho_{ij}(t))_{i,j=1,\dots,N,\,i<j}^\top$ (see~\ref{H1}) and $E(t)$ is the time-dependent square matrix defined row-wise as follows: let us label each of the $N(N-1)/2$ rows by $e^{ij}\in\R^{N(N-1)/2}$, where $i,j=1,\dots,N,\,i<j$, i.e.
\begin{equation}
e^{ij}(t)=\big(e^{ij}_{lk}(t)\big)_{l,k=1,\dots,N,\,l<k};
\end{equation}
then, fixed $i,j=1,\dots,N,\,i<j$, we have that for $l,k=1,\dots,N,\,l<k$,
\begin{equation}\label{eq:entries_E}
e^{ij}_{lk}(t)=\begin{cases}
2\delta_{ij}(t) &\text{if }l=i \text{ and } k=j,\\[1ex]
\eta\big(a_{jk}(t)-a_{ik}(t)\big)&\text{if }l=i\text{ and } k\neq j,\\[1ex]
\eta\big(a_{ik}(t)-a_{jk}(t)\big)&\text{if }l=j,\\[1ex]
0&\text{otherwise}.
\end{cases}
\end{equation}
Note that the entries of $E$ outside the diagonal are either positive or null. 
In other words, the coefficients obtained through the inequality in Lemma~\ref{lem:error_ij_diff_inequality}, have been further bounded from above in the sense that only the terms with positive sign are left.
Moreover, due the inequality in~\eqref{eq:cond_sync_1}, the entries on the diagonal, that is $e^{ij}_{ij}(t)=2\delta_{ij}(t)$, are strictly negative and the matrix is row-dominant \cite[Definition 7.10]{fink1974almost} for $t\ge t_0$ .
Consequently, the linear homogeneous problem $\dot u=E(t)u$ has dichotomy spectrum contained in $(-\infty,0)$ for $t\ge t_0$ (and thus it admits an exponential dichotomy with projector the identity on $[t_0,\infty)$)~\cite[Theorem 7.16)]{fink1974almost}. In fact a stronger property of exponential decay holds, i.e., denoted by $U(t,s)$ the principal matrix solution of $\dot u=E(t)u$ at $s\in\R$, it holds that,
\begin{equation}\label{eq:13/11-16:46}
\|U(t,s)\|\le e^{-\overline\gamma(t-s)},\quad \text{for all } t\ge s\ge t_0.
\end{equation}
Incidentally, this means that the solution of the non-homeogeneous linear problem $u'= E(t)u+\beta(t)$, $u(t_0)=\xi(t_0)$ which is given by the variation of constants formula
\begin{equation}\label{eq:18/11-11:17}
u\big(t,t_0,\xi(t_0)\big)=U(t,t_0)\xi(t_0)+\int_{t_0}^tU(t,s)\beta(s)\,ds,   
\end{equation}
where the integral is understood component-wise, is defined for all $t\ge t_0$ thanks to~\eqref{eq:13/11-16:46}.

We  shall prove that~\eqref{eq:control} implies $\xi(t)\le u\big(t,t_0,\xi(t_0)\big)$ for all $t\ge t_0$. Considering $\ep>0$ and $\overline \ep= \ep (1,\dots,1)\in\R^{N(N-1)/2}$, note that from the continuity with respect to initial data, it is enough to prove that $\xi(t)\ll u\big(t,t_0,\xi(t_0)+\overline \ep\big)=:u(t, \ep)$  for all $t\ge t_0$. 
Assume, on the contrary, that there is a first time $t_1>t_0$ for which the equality holds for some component. By simplicity of notation, let the first be such a component. Then,
\begin{equation}\label{eq:18/11-12:04}
\xi_{ij}(t)< u_{ij}(t,\ep)\quad\text{for }t\in[t_0,t_1) \qquad\text{and}\qquad \xi_{12}(t_1)= u_{12}(t_1, \ep).
\end{equation}
Denote by $g_\xi,g_u:\R\times\R\to\R$, the Carathéodory functions defined by
\begin{equation}\begin{split}
g_\xi(t,v)=2\delta_{12}(t)v &+\sum_{\substack{k=1\\k\neq 1,2}}^N\Big[
\eta\big(a_{2k}(t)-a_{1k}(t)\big)\xi_{1k}(t)+\eta\big(a_{1k}(t)-a_{2k}(t)\big)\xi_{2k}(t)\Big]
\end{split}
\end{equation}
\begin{equation}\begin{split}
g_u(t,v)=2\delta_{12}(t)v&+\!\sum_{\substack{k=1\\k\neq 1,2}}^N\!\Big[
\eta\big(a_{2k}(t)-a_{1k}(t)\big)u_{1k}(t,\ep)+\eta\big(a_{1k}(t)-a_{2k}(t)\big) u_{2k}(t,\ep)\Big] 
\end{split}
\end{equation}
and consider the scalar Carathéodory differential problems $\dot v=g_\xi(t,v)$, $\dot v=g_u(t,v)$. Due to Lemma~\ref{lem:error_ij_diff_inequality} and the assumptions in~\eqref{eq:18/11-12:04}, the following scalar differential inequalities holds true,
\begin{equation}
\dot\xi_{12}(t)\le g_\xi\big(t,\xi_{12}(t)\big)\le g_u\big(t,\xi_{12}(t)\big),\quad\text{for all }t\in[t_0,t_1].
\end{equation}
Hence, the comparison theorem for Carath\'{e}odory scalar differential equations (see Olech and Opial~\cite{olech1960inegalite}) yields
\begin{equation}
\xi_{12}(t)\le v\big(t,t_0,\xi_{12}(t_0)\big)<v\big(t,t_0,\xi_{12}(t_0)+\ep\big)=u_{12}(t,\ep)
\end{equation}
for every $t\in[t_0,t_1]$ and in particular for $t=t_1$. However, this contradicts~\eqref{eq:18/11-12:04}. Hence, it must be that such $t_1\in\R$ does not exist and $\xi(t)\ll u\big(t,t_0,\xi(t_0)+\overline \ep\big)$  for all $t\ge t_0$. 
In turn, the arbitrariness on $\ep>0$ gives us the sought for ordering of vector solutions $\xi(t)\le u\big(t,t_0,\xi(t_0)\big)$ for all $t\ge t_0$. Then, from~\eqref{eq:18/11-11:17} we immediately obtain that
\begin{equation}\label{eq:28/11-16:41}
|\xi(t)|\le \|U(t,t_0)\|\,|\xi(t_0)|+\int_{t_0}^t\|U(t,s)\|\,|\beta(s)|\,ds.
\end{equation}
Concerning the first term, we have that $\|U(t,t_0)\|\,|\xi(t_0)|\le 4\rho^2e^{-\overline \gamma(t-t_0)}$, thanks to~\eqref{eq:13/11-16:46} and~\ref{H2}. Therefore, for any $\ep>0$, it holds 
\begin{equation}\label{eq:28/11-16:42}
\|U(t,t_0)\|\,|\xi(t_0)|<\ep,\qquad\text{whenever } t-t_0>\frac{1}{\overline\gamma}\ln\left(\frac{4\rho^2}{\ep}\right).
\end{equation}
We shall, thus, analyze the second term, the integral $\int_{t_0}^t\|U(t,s)\|\,|\beta(s)|\,ds$. Note that, thanks to~\ref{H3}, $\mu_1:=\sup_{\tau\in\R}\int_\tau^{\tau+1}|\beta(s)|\, ds<\infty$. Then, we have that 
\begin{equation} \label{eq:VOC_nonhomogeneous_integral_estimate}
\begin{split}
\int_{t_0}^t\|U(t,s)\|\, |\beta(s)|\,ds&\le \int_{t_0}^t |\beta(s)|e^{-\overline\gamma(t-s)}\,ds= \int_{0}^{t-t_0} \!\!\!|\beta(t-u)|e^{-\overline\gamma u}\,du\\
&\le \int_{0}^{\infty}\!\! |\beta(t-u)|e^{-\overline\gamma u}\,du
\le \sum_{n=0}^\infty\int_{n}^{n+1}\!\!\!\!|\beta(t-u)|e^{-\overline\gamma u}\,du \\
&\le  \sum_{n=0}^\infty e^{-\overline\gamma n}\int_{n}^{n+1}\!\!\!|\beta(t-u)| \,du=\frac{\mu_1 }{1-e^{-\overline\gamma}}.
\end{split}
\end{equation}
Furthermore, since by assumption $\overline\gamma>-\log(1-\mu_1/M)$, 
then
\begin{equation}
\frac{\mu_1 }{1-e^{-\overline\gamma}}<M.
\end{equation}
This inequality, together with~\eqref{eq:28/11-16:41} and~\eqref{eq:28/11-16:42}
gives the sought-for result.
\end{proof}

\begin{remark}
    A closer look at the proof of Theorem~\ref{thm:sync_up_to_constant} reveals that the fundamental step was to show that, under the given assumptions, the homogeneous linear system $\dot u =E(t) u$ has dichotomy spectrum contained in $(-\infty,0)$. Yet, the property of row-dominance is only a sufficient condition for the existence of an exponential dichotomy with projector the identity. 
    Therefore, it is worth noting that, although we privileged row-dominance in order to give a set of easily computable inequalities, other sufficient conditions may be just as effective in guaranteeing synchronization. 
    For example see also Proposition 1.5 in \cite{coppel2006dichotomies} which uses the time-dependent eigenvalues.
    Another option is the weaker requirement that the Lyapunov spectrum is contained in $(-\infty,0)$. Numerical methods to approximate the Lyapunov and the dichotomy spectrum under the assumption of integral separation can be found in~\cite{dieci2002lyapunov2,dieci2002lyapunov,dieci2007lyapunov,froyland2013computing}. See also Remark~\ref{rmk:ED} for further details. 
\end{remark}

We notice that in our main Theorem~\ref{thm:sync_up_to_constant}, the synchronization error might be largely influenced by the heterogeneity of the nodes. In the next corollary we address such an issue provided that the network has a global coupling strength. 


\begin{corollary}\label{cor:global_coupling}
    Suppose that Theorem~\ref{thm:sync_up_to_constant} holds for~\eqref{eq1}. Then, the same Theorem holds for the system with global coupling strength
    \begin{equation}\label{eq1g}
        \dot x_i = f_i(t,x_i) + c\sum_{k=1}^N a_{ik}(t)(x_k-x_i), \qquad c\geq1
\end{equation}
but with a synchronization error
\begin{equation}
|\sigma_i(t)-\sigma_j(t)|^2< \ep+\frac{1}{c}M.
\end{equation}
\end{corollary}
\begin{proof}
    The proof is a slight adaptation of the proof of Theorem~\ref{thm:sync_up_to_constant}, therefore, we refer to the notation used there. First, let  $\alpha_{ij}\leq0$ in~\ref{H1}. In such a case, it suffices to consider $\alpha_{ij}=0$. It follows as in the proof of Theorem~\ref{thm:sync_up_to_constant} that the error dynamics $\dot\xi(t)$ is an under function with respect to the initial value problem $u'=cE(t)u+\beta(t)$. In turn, this initial value problem is smoothly equivalent, via time rescaling, to $u'=E(t)u+\frac{1}{c}\beta(t)$. The rest of the proof follows  by the same arguments as in the proof of Theorem~\ref{thm:sync_up_to_constant} but with $\beta(t)$ replaced by $\frac{1}{c}\beta(t)$.
    Next, we consider the case when $\alpha_{ij}>0$. Since we assume that Theorem~\ref{thm:sync_up_to_constant} holds for $c=1$, it necessarily follows that $\Big(a_{ij}(t)+a_{ji}(t)
    +\frac{1}{2}\sum_{\substack{k=1\\k\neq i,j}}^N\big(a_{jk}(t)+a_{ik}(t)\big)\Big)>0$. Let 
    \begin{equation}
    \begin{split}
    \bar\delta_{ij}&\coloneqq\alpha_{ij}-c\Big(a_{ij}(t)+a_{ji}(t)
    +\frac{1}{2}\sum_{\substack{k=1\\k\neq i,j}}^N\big(a_{jk}(t)+a_{ik}(t)\big)\Big)\\    
    &=c\bigg( \frac{\alpha_{ij}}{c}-\Big(a_{ij}(t)+a_{ji}(t)
    +\frac{1}{2}\sum_{\substack{k=1\\k\neq i,j}}^N\big(a_{jk}(t)+a_{ik}(t)\big)\Big) \bigg)\leq c\delta_{ij}.
    \end{split}
    \end{equation}
    By the previous inequality, we now have the same arguments as above for the error dynamics $\dot\xi_{ij}(t)$.
\end{proof}

Analogous results to Corollary~\ref{cor:global_coupling} are known, for example, for the static case with $a_{ij}\geq0$~\cite{pereira2011stability} provided that the internal dynamics are dissipative. Nevertheless, our result also covers the time-varying case and allows for the weights to be negative.

Let us next present a corollary that might be of additional use in practical cases. 
Assume that the function $\beta$ in~\ref{H3} is in fact essentially bounded. Then a stronger result is available that highlights the role of the coupling strength in the network. Particularly, a result of sharp synchronization appears when $\|\beta(\cdot)\|_{L^\infty}=0$.

\begin{corollary}\label{cor:sharp_sync}
Under the assumptions of {\rm Theorem~\ref{thm:sync_up_to_constant}}, if additionally $\beta(\cdot)\in L^\infty$, then, for every $\ep>0$ there is $T(\ep)=\frac{1}{\gamma}\ln\left(4\rho^2/\ep\right)>0$ such that for all $i,j=1,\dots,N$
\begin{equation}
|\sigma_i(t)-\sigma_j(t)|^2< \ep+\frac{\|\beta(\cdot)\|_{L^\infty}}{\overline \gamma},\qquad\text{for } t-t_0>T(\ep).
\end{equation}
Moreover, if the considered system has a global coupling coefficient $c>0$ as in~\eqref{eq1g}, then the previous inequality reads as 
\begin{equation}
|\sigma_i(t)-\sigma_j(t)|^2< \ep+\frac{\|\beta(\cdot)\|_{L^\infty}}{c\,\overline \gamma},\qquad\text{for } t-t_0>T(\ep).
\end{equation}
If $f_i=f$ for all $i=1,\dots,N$, and~\ref{H2} holds with $\mu=0$, then~\eqref{eq1} synchronizes.
\end{corollary}
\begin{proof}
The result is an easy consequence of Theorem~\ref{thm:sync_up_to_constant} once~\eqref{eq:VOC_nonhomogeneous_integral_estimate} is changed for 
\begin{equation}
\begin{split}
\int_{t_0}^t\|U(t,s)\|\, |\beta(s)|\,ds&\le \|\beta(\cdot)\|_{L^\infty} \int_{0}^{t-t_0} e^{-\overline\gamma u}\,du=\frac{\|\beta(\cdot)\|_{L^\infty}}{\overline \gamma}\big(1-e^{-\overline \gamma(t-t_0)}\big).
\end{split}
\end{equation}
Then, the case of global coupling is inherited from the proof of Corollary ~\ref{cor:global_coupling}. Finally, the result of sharp synchronization for a network of identical agents follows immediately by taking $\beta_{ij}(t)=0$ for all $t\in\R$ and all $i,j=1,\dots,N$ (see also Remark~\ref{rmk:f_i=f_j implies_beta=0}).
\end{proof}

Our last corollary is concerned with the strengthening of \ref{H1} from a local to a global property, i.e.~when the functions $\alpha^r_{ij},\beta^r_{ij}\in L^1_{loc}(\R,\R)$ do not depend on $r>0$. Note that this is for example the case when identical nodes with a global Lipschitz constant are chosen.
\begin{enumerate}[label=\upshape(\textbf{H1*}),leftmargin=33pt,itemsep=2pt]
\item\label{H1S} Assume that there are functions $\alpha_{ij},\beta_{ij}\in L^1_{loc}(\R,\R)$, with $\beta_{ij}$ non-negative, such that \ref{H1} holds with $\alpha^r_{ij}=\alpha_{ij}$ and $\beta^r_{ij}=\beta_{ij}$ for every $r>0$ and every $i,j=1,\dots,N$.
\end{enumerate}
It is immediate to check that an analogous version of Lemma \ref{lem:two_nodes_assumptions_from_points_to_functions} holds true also for \ref{H1S}.

If \ref{H1S} is in force, then the assumption \ref{H2} can be weakened in the sense that boundedness for the solution $\sigma$ in \ref{H2} becomes dispensable.
\begin{enumerate}[label=\upshape(\textbf{H2*}),leftmargin=33pt,itemsep=2pt]
\item\label{H2S} 
Assume that \ref{H2} holds, except the solution $\sigma(t)=(\sigma_1(t),\dots,\sigma_N(t))^\top\in\U_t$~\eqref{eq1}, is not necessarily bounded.
\end{enumerate}

\begin{corollary}
    Assume that \ref{H1S}, \ref{H2S} and \ref{H3}  hold. Then with the respective additional assumptions, the results of {\rm Theorem \ref{thm:sync_up_to_constant}}, and {\rm Corollaries \ref{cor:global_coupling}} and  {\rm \ref{cor:sharp_sync}} still hold true.
\end{corollary}
\begin{proof}
    The results are immediate once one notices that the boundedness of $\sigma(t)$ in \ref{H2} was only needed to fix $\rho>0$ so that \ref{H1} could be used in Lemma \ref{lem:error_ij_diff_inequality} for the suitable $\alpha^\rho_{ij},\beta^\rho_{ij}\in L^1_{loc}(\R,\R)$. The same argument of  Lemma \ref{lem:error_ij_diff_inequality} can now be repeated to construct a differential inequality holding for any (possibly not bounded) solution of \eqref{eq1}. 
\end{proof}

To highlight some the results of this section, let us present a couple of examples.

\begin{example}[Synchronization of heterogeneous van der Pol oscillators with time-dependent perturbation parameter]\label{ex:vdp}

Let us consider a network of $N\geq2$ heterogeneous van der Pol oscillators. Each oscillator has internal dynamics given by
\begin{equation}
\begin{split}
   u_i'&= v_i+b_iu_i-\displaystyle\frac{u_i^3}{3}\\
    v_i'&=-\ve_i(t)u_i.  
\end{split}
\end{equation}

Notice that since the individual oscillators are heterogeneous, both in amplitude and phase, we expect accordingly that $\alpha_{ij}>0$ and $\beta_{ij}>0$ in \ref{H1}. Nevertheless, in this example, we shall show that a strong enough coupling allows us to synchronize the oscillators. This observation goes in hand with \ref{H2} since for a weak enough coupling we indeed expect the existence of a large enough inflowing invariant ball for the coupled dynamics. Finally, since the individual nodes are all oscillators, they differ on bounded quantities, hence \ref{H3} holds. 

For this example, we randomly choose the parameters $b_i\in(\frac{1}{2},1)$ from a uniform distribution, which influence the amplitude of each oscillator. Furthermore, we let $\ve_i(t)=\ve_0\left(1+\frac{1}{2}\sin(\omega_it) \right)$, with $\omega_i>0$ also picked uniformly at random in the interval $\omega_i\in(1,2)$, and $0<\ve_0\ll1$. Naturally, we can choose any other time-varying behavior of $\ve(t)$ as long as it is positive but sufficiently small, i.e., $0<\ve(t)\ll1$ for almost all $t\geq t_0$. Regarding the network, we choose a piece-wise constant adjacency matrix $A(t)=[a_{ij}(t)]$ where $a_{ij}\in\left\{ 0,1\right\}$ updates randomly every $\Delta t$ units of time (we do ensure that the underlying graph is always connected). Finally, we also assume that the interconnection occurs on the fast timescales, meaning that the model we consider reads as
\begin{equation}\label{eq:vdp}
    \begin{split}
         u_i'&= v_i+b_iu_i-\displaystyle\frac{u_i^3}{3}+\frac{c}{N}\sum_{j=1}^Na_{ij}(t)(u_j-u_i)\\
    v_i'&=-\ve_i(t)u_i +\frac{c}{N}\sum_{j=1}^Na_{ij}(t)(v_j-v_i).
    \end{split}
\end{equation}

We notice that since the weights $a_{ij}$ are nonnegative and the dynamics of the decoupled nodes are ultimately bounded, one can guarantee that Theorem~\ref{thm:sync_up_to_constant}, and especially Corollary~\ref{cor:global_coupling}, hold for $c$ large enough. In Figure~\ref{fig:vdp1} we show a corresponding simulation for $N=5$ and $\Delta t=50$, while in Figure~\ref{fig:vdp2} we show the results for a similar setup, but with $N=100$ and $\Delta t=5$ (for practicality we only show the error dynamics for the latter one). In both cases we compare the effect of the global coupling $c$, as described in Corollary~\ref{cor:global_coupling} and verify that the synchronization error decreases as $c$ increases. The shown pictures are representative among ten different simulations, all of which show a similar qualitative behavior.
\begin{figure}[htbp]
    \centering
    \includegraphics{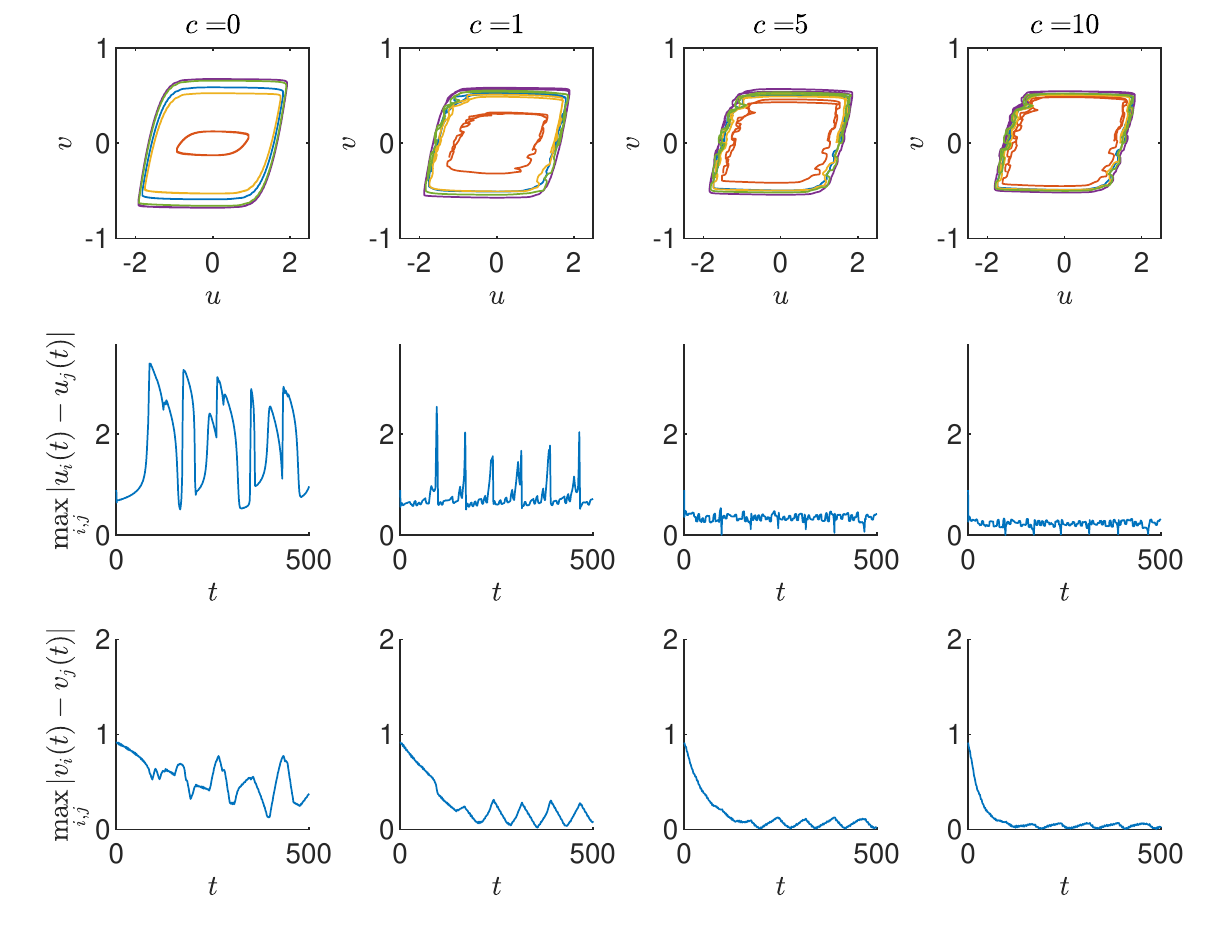}
    \caption{A temporal network of $N=5$ heterogeneous van der Pol oscillators, and update time $\Delta t=50$, as described above. Each column corresponds to a particular value of $c$ in~\eqref{eq:vdp}. In the first row, we plot the projection to the $(u,v)$-plane of the solutions. Since~\eqref{eq:vdp} is nonautonomous, the apparent intersections are just due to the projection. The second and third rows show the maximum of pairwise errors for each component. According to Corollary~\ref{cor:global_coupling}, there is a large enough $c$ leading to synchronization. This is verified as one goes from left to right in the plots. We notice that the attractor, for example in the right most column, does not appear regular due to the time-dependent random switching of the network topology.}
    \label{fig:vdp1}
\end{figure}
\begin{figure}[htbp]
    \centering
    \includegraphics{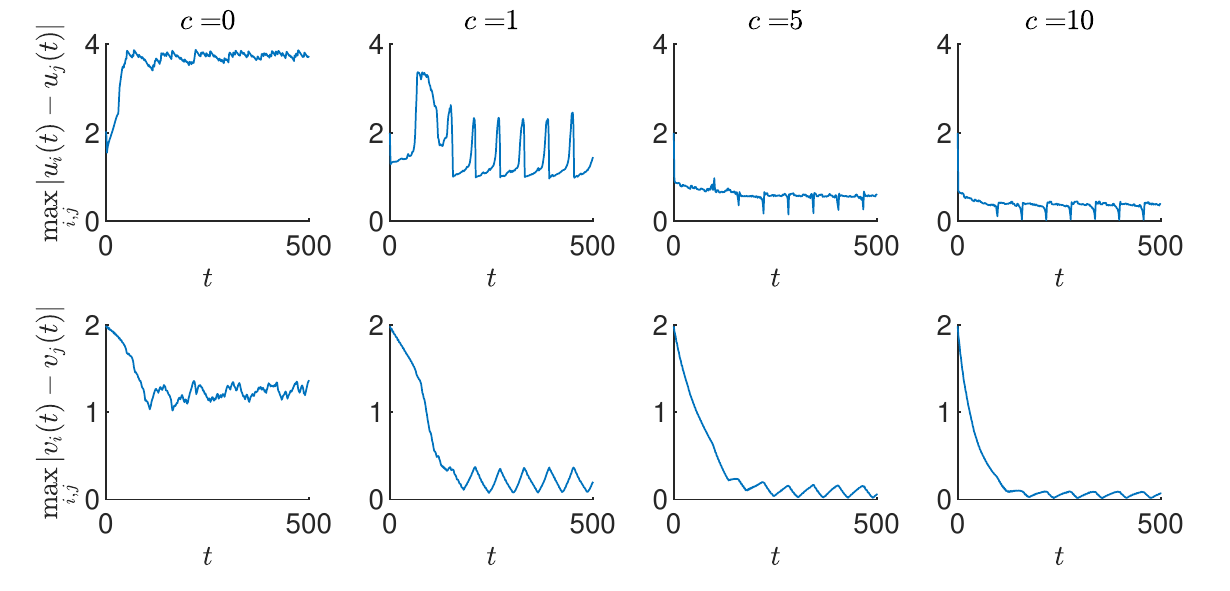}
    \caption{Error plots analogous to those in Figure~\ref{fig:vdp1}, corresponding to a simulation with $N=100$ and $\Delta t= 5$. Observe that, again, as $c$ increases the synchronization error decreases, as guaranteed by Corollary~\ref{cor:global_coupling}.}
    \label{fig:vdp2}
\end{figure}
\end{example}

In the next example, we exploit Theorem~\ref{thm:sync_up_to_constant} to achieve synchronization in the case when some weights in the network can be negative. We also briefly argue how our synchronization results on a static network may persist under small enough time-varying perturbations of the weights. The latter case is further detailed in section~\ref{sec:time_dependent_perturbations}.

\begin{example}[Compensating contrarians for consensus on a ring network] 

Consensus dynamics is very important in several fields of science~\cite{becchetti2020consensus,spanos2005dynamic,tahbaz2006consensus}. Let us consider a ring network with $N$ nodes as sketched in Figure~\ref{fig:ring}.
\begin{figure}[htbp]
    \centering
    \begin{tikzpicture}
    \graph[counterclockwise, radius=2.5cm, n=10]
 {1[as=$x_1$],2[as=$x_2$],3[as=$x_3$],a[as=$x_{i-2}$],i[as=$x_{i-1}$],j[as=$x_i$],k[as=$x_{i+1}$],b[as=$x_{i+2}$],M/[as=$x_{N-1}$],N/$x_N$;
    };
    \draw[->,red] (1) to[bend left=10](N);
    \draw[->,red] (1) to[bend right](M);
    \draw[->,red] (1) to[bend right=10](2);
    \draw[->,red] (1) to[bend left](3);
    \draw[->] (2) to[bend right=15](3);
    \draw[->] (3) to[bend right=15](2);
    \draw[->] (N) to[bend right=15](M);
    \draw[->] (M) to[bend right=15](N);
    \draw[dotted] (3) to[](a);
    \draw[dotted] (b) to[](M);
    \draw[->] (i) to[bend right=15](j);
    \draw[->] (j) to[bend right=15](i);
    \draw[->] (j) to[bend right=15](k);
    \draw[->] (k) to[bend right=15](j);
    \draw[->] (k) to[bend right=15](b);
    \draw[->] (b) to[bend right=15](k);
    \draw[->] (a) to[bend right=15](i);
    \draw[->] (i) to[bend right=15](a);

    \draw[->] (a) to[bend right=75](j);
    \draw[->] (j) to[bend right=25](a);

    \draw[->] (b) to[bend right=25](j);
    \draw[->] (j) to[bend right=75](b);

    \draw[->] (2) to[bend right=25](N);
    \draw[->] (N) to[bend right=75](2);
    \end{tikzpicture}
    \caption{A ring network where, \emph{except for the contrarian node $x_1$}, all nodes interact in a bidirectional way with its nearest $2$-neighbors. The contrarian node $x_1$ has a negative influence on its $2$-neighbors, while the neighbors do not influence $x_1$. In this example we assume that all the conformist weights (black arrows) are $1$, while the contrarian weights (red arrows) are $-a$, $a>0$. In this setup, \emph{consensus in not achieved}. We use the results of Theorem~\ref{thm:sync_up_to_constant} to determine what influence $x_2$ must have on $x_1$ to reach consensus.}
    \label{fig:ring}
\end{figure}
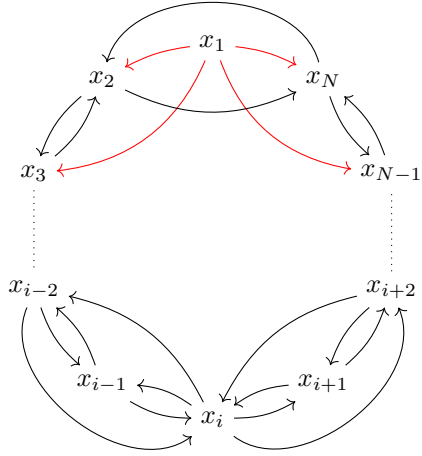
In this example, each node is scalar $x_i\in\R$ and interacts with its nearest $2$-neighbors. We assume that the network dynamics are governed by the well-known consensus protocol $\dot x=-Lx$, where $L$ denotes the (signed) Laplacian. Component-wise, the dynamics of each node is determined by $\dot x_i=\sum_{j=i-2}^{i+2}a_{ij}(x_j-x_i)$ with $j\in\mathbb Z\text{ mod }N$, $a_{ii}=0$, and where $a_{ij}\neq0$ denotes a connection from node $j$ towards node $i$. From now on, we stick to the previously described $2$-nearest neighbor topology.

Motivated by~\cite{PhysRevE.84.046202}, we identify $a_{ij}>0$ with a ``conformist'' influence, while $a_{ij}<0$ is referred to as a ``contrarian'' influence. Let us consider that there is one contrarian node. Without loss of generality, let node $x_1$ be the contrarian represented as $a_{i1}<0$ for $i=1,2,N-1,N$. Moreover, we assume for now that neighboring nodes to $x_1$ do not influence it, that is $a_{1j}=0$ for $j=1,2,N-1,N$, and that the rest of the nodes are conformists, i.e., the remaining nonzero weights $a_{ij}$, are positive. We recall that in the case where all the $a_{ij}$'s are positive, the dynamics of the consensus protocol lead to convergence of the solutions to some finite value. In the case of a signed Laplacian, where negative weights are allowed, determining the stability of the protocol is considerably harder. In particular, in the above setting consensus does not hold. In this example we are going to use Theorem~\ref{thm:sync_up_to_constant} to determine how strongly would node $x_2$ need to influence node $x_1$ (that is choose $a_{12}$) to overcome $x_1$'s negative influence so that the protocol can reach consensus. 

For simplicity, let $a_{i1}=-a$, $a>0$, and $a_{ij}=1$ for $i,j\geq2$. With the notation of Theorem~\ref{thm:sync_up_to_constant}, we have that $\alpha_{ij}=0$ and $\mu_1=0$. Moreover, notice that for all $i,j\in\left\{4, \ldots, N-2 \right\}$ one has $\delta_{ij}<0$. Accounting for the symmetry $\delta_{ij}=\delta_{ji}$ we have
\begin{equation}
    \begin{split}
        \delta_{12} &=-\left( a_{12}+\frac{3}{2}-a \right)\\
        \delta_{13}=\delta_{1N}=\delta_{1(N-1)}&=-\left(\frac{a_{12}}{2}+\frac{3}{2}-a\right)\\
        \delta_{23}=\delta_{2N}&=-\left(\frac{5}{2}-a \right).
    \end{split}
\end{equation}
Since $a_{12}$ has no influence on $\delta_{23}$, we impose the further constraint that $a<\frac{5}{2}$. On the other hand, regarding $\gamma$ let
\begin{equation}
    \gamma_{ij}=2|\delta_{ij}|-\sum_{k\neq i,j}|a_{jk}-a_{ik}|.
\end{equation}
So, accounting for the symmetry mentioned above, we look to satisfy the inequalities
\begin{equation}
    \begin{split}
        \gamma_{12}
        &=2\left|  a_{12}+\frac{3}{2}-a\right|-3>0\\
        \gamma_{13}
        &=2\left| \frac{a_{12}}{2}+\frac{3}{2}-a \right|-|1-a_{12}|-2>0\\
        \gamma_{23}
         &=2\left|\frac{5}{2}-a \right|-2>0,
    \end{split}
\end{equation}
where, as for the $\delta_{ij}$'s above, the rest of the $\gamma_{ij}$'s are all nonnegative. We notice that $\gamma_{12}>0$ can always be satisfied by an appropriate choice of $a_{12}$ and that since $a_{12}$ has no influence on $\gamma_{23}$, we have the further restriction $a<\frac{3}{2}$. However $\gamma_{13}>0$ further imposes that $a<1$. With this restriction, $\delta_{ij}<0$ for any $a_{12}>0$. So, from Theorem~\ref{thm:sync_up_to_constant}, we can conclude that for $0<a<1$ the given consensus protocol achieves consensus provided that the contrarian influence is compensated by $a_{12}>a$. This last inequality is obtained from the requirement $\gamma_{12}>0$. 

It is evident from the analysis performed in this example, that the weights $a_{ij}$ can be time-varying. For more general details see section~\ref{sec:time_dependent_perturbations}. Moreover, the contrarian influence does not have to be homogeneous, the result above will hold as long as all of them have modulus less than $1$. In Figure~\ref{fig:contrarians} we show a couple of simulations that verify our arguments with the conformist weights fixed to $1$, the contrarian weights given by $a_{i1}=a_{i1}(t)=-\frac{1}{2}+\frac{1}{2}\sin(\omega_it)$ with uniformly distributed random frequencies, and the conformist compensation $a_{12}=0$, on the left, and $a_{12}=1$ on the right.


\begin{figure}[htbp]
    \centering
   \includegraphics[width=0.403\textwidth]{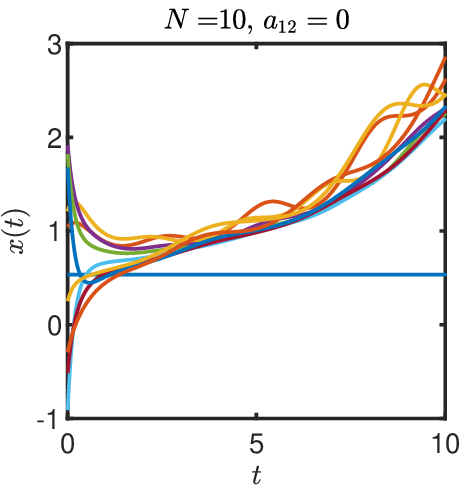}\hspace{0.5cm}
    \includegraphics[width=0.403\textwidth]{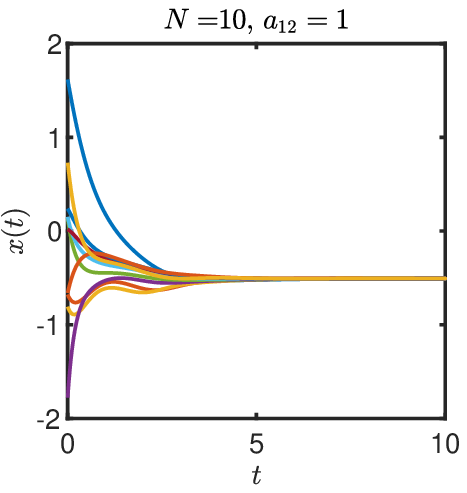}\\    \hspace{-0.5cm}\includegraphics[width=0.415\textwidth]{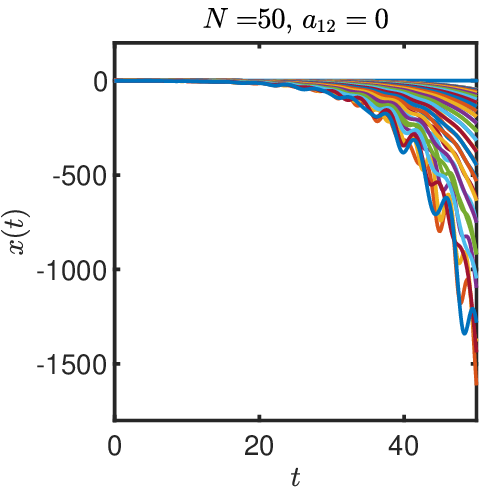}\hspace{0.85cm}
    \includegraphics[clip,trim={0cm -0.1cm 0cm 0cm},width=0.405\textwidth]{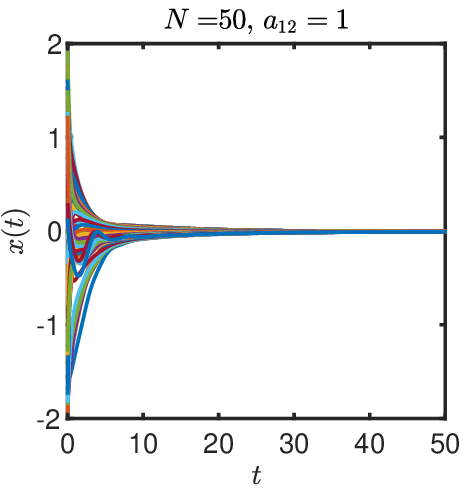}
    \caption{Simulations of the consensus protocol $\dot x=-Lx$, for a ring network as in Figure~\ref{fig:ring}. For all simulations, the initial conditions are randomly chosen, the conformist weights are set to $1$, and the contrarian weights are $a_{i1}=-\frac{1}{2}+\frac{1}{2}\sin(\omega_it)$, $i=2,3,N-1,N$, for some randomly chosen frequency $\omega_i>0$. In the left we show the dynamics where no compensation is implemented, i.e.~$a_{12}=0$. One can clearly notice that not only the overall dynamics is unstable (due to the negative weights), but there is no consensus at all. On the right we verify that from Theorem~\ref{thm:sync_up_to_constant} we have deduced that, for example, a weight $a_{12}=1$ leads to consensus.}
    \label{fig:contrarians}
\end{figure}

\end{example}

\section{On the synchronization of a cluster}\label{sec:cluster}
In practical cases, the synchronization of the whole network can be sometimes excessive and one is rather interested in the behaviour of a limited portion of the network. Hereby, we show how our arguments can be used to analyze also such case.

\begin{theorem}\label{thm:sync_cluster}
Let $2<n\le N$ integer and consider an ordered set $J$ of $n$ nodes from~\eqref{eq1} identified by their indices, i.e.~$J=\{j_1,\dots,j_n\}\subset\{1,\dots,N\}$, with $j_m<j_{m+1}$ for all $m=1,\dots,n-1$. Assume that ~\ref{H2} holds for~\eqref{eq1}, and that~\ref{H1} and~\ref{H3} are true for all $i,j\in J$. 

Furthermore, assume that $\{a_{jk}(\cdot+t)-a_{ik}(\cdot+t)\in L^1_{loc}\mid t\in\R,\, i,j\in J, k\notin J\}$ is $L^1_{loc}$-bounded, and call 
\begin{equation}
0\le\mu_2:=2\rho^2\max_{\substack{i,j\in J,\\ k\notin J}}\sup_{\tau\in\R}\int_\tau^{\tau+1}|a_{jk}(s)-a_{ik}(s)|\, ds<\infty.
\end{equation}
Fixed any $M>\mu_1+\mu_2(N-n)\sqrt{2n(n-1)}$, if there is $t_0\in\R$ such that for all $i,j\in J$ with $i<j$ and almost every $t>t_0$,
\begin{equation}
\begin{split}
\delta_{ij}(t):= \alpha^\rho_{ij}(t)-\Big(a_{ij}(t)+a_{ji}(t)
    +\frac{1}{2}\sum_{\substack{k=1\\k\neq i,j}}^N\big(a_{jk}(t)+a_{ik}(t)\big)\Big)<0\\
\end{split}    
\end{equation}
and 
\begin{equation}
\begin{split}
\overline \gamma_J:=\inf_{t\in\R}\min_{\substack{i,j=1,\dots,N,\\ i\neq j}}\Big\{2|\delta_{ij}(t)|&-\sum_{k\in J\setminus \{i,j\}}\big|a_{jk}(t)-a_{ik}(t)\big|\Big\}\\
&>-\log\left(1-\frac{\mu_1+\mu_2(N-n)\sqrt{2n(n-1)}}{M}\right),
\end{split}
\end{equation}
then for every $\ep>0$ there is $T(J,\ep)=\frac{1}{\overline \gamma_J}\ln\left(4\rho^2/\ep\right)>0$ such that 
\begin{equation}
|\sigma_i(t)-\sigma_j(t)|^2< \ep+M,\qquad\text{for } t-t_0>T(J,\ep).
\end{equation}
In other words, $J$ synchronizes up to a constant in finite time.
\end{theorem}
\begin{proof}
For the sake of simplicity, we shall assume that $J=\{1,\dots,n\}$. Then, we can proceed akin to the proof of Theorem~\ref{thm:sync_up_to_constant} but now considering the vector $\zeta(t):=\big(|\sigma_i(t)-\sigma_j(t)|^2\big)_{i,j\in J,\,i<j}^\top$. 
Thanks to Lemma~\ref{lem:error_ij_diff_inequality}, we have that $\dot\zeta(t)$ is an \emph{under-function} with respect to the initial value problem $u'= C(t)u+\nu_J(t)$, $u(t_0)=\zeta(t_0)$, i.e.,
\begin{equation}
\dot \zeta(t)\le C(t)\zeta(t)+\nu_J(t),\quad\text{for all }t>t_0,\,t_0\in\R,
\end{equation}
where $\nu_J(t)=(\nu_{ij}(t))_{i,j\in J,\,i<j}^\top$ is defined by
\begin{equation}\label{eq:nu_cluster}
\nu_{ij}(t)=\eta\left(2\beta^\rho_{ij}(t)+\sum_{k\notin J}\big(a_{jk}(t)-a_{ik}(t)\big)\,\big(\xi_{ik}(t)-\xi_{jk}(t)\big)\right),
\end{equation}
and $C(t)$ is the time-dependent matrix constructed exactly as the matrix $E(t)$ in the proof of Theorem~\ref{thm:sync_up_to_constant} where now $n$ appears in place of $N$ everywhere. Now, note that, under the given assumptions, $\dot u=C(t)u$ admits an exponential dichotomy on $[t_0,\infty)$ with projector the identity and exponential rate of convergence $\gamma_J>0$; as for Theorem~\ref{thm:sync_up_to_constant}, the linear homogeneous system $\dot u=C(t)u$ is row-dominant for $t\ge t_0$~\cite[Theorem 7.16)]{fink1974almost}. Then, with analogous reasoning, we obtain that for all $t\ge t_0$,
\begin{equation}
\begin{split}
\big(|\sigma_i(t)-&\sigma_j(t)|^2\big)_{i,j=1,\dots,n,\,i<j}^\top\\
\le &U(t,t_0)\big(|\sigma_i(t_0)-\sigma_j(t_0)|^2\big)_{i,j=1,\dots,n,\,i<j}^\top+\int_{t_0}^tU(t,s)\nu(s)\,ds,
\end{split}
\end{equation}
where $U(t,t_0)$ is the principal matrix solution at $t_0$ of $\dot u=C(t)u$. Then, fixed $\ep>0$ and reasoning as in the proof of Theorem~\ref{thm:sync_up_to_constant}, but recalling that $\nu$ is defined by~\eqref{eq:nu_cluster}, we come to the analogous chain of inequalities as~\eqref{eq:VOC_nonhomogeneous_integral_estimate}, that is,
\begin{equation}
\begin{split}
\int_{t_0}^t&\|U(t,s)\|\, \Big(|\beta(s)|+\sqrt{\frac{n(n-1)}{2}}\max_{i,j\in J}\sum_{k\notin J}|a_{jk}(s)-a_{ik}(s)|\,|\xi_{ik}(t)-\xi_{jk}(t)|\Big)\,ds\\
&\le \int_{0}^{t-t_0} \!\!\Big(|\beta(t-u)|+4\rho^2\sqrt{\frac{n(n-1)}{2}}\max_{i,j\in J}\sum_{k\notin J}|a_{jk}(t-u)-a_{ik}(t-u)|\Big)e^{-\overline\gamma_J u}\,du\\
&\le  \sum_{n=0}^\infty e^{-\overline\gamma_J n} \big(\mu_1+\mu_2(N-n)\sqrt{2n(n-1)}\big)=\frac{\mu_1 +\mu_2(N-n)\sqrt{2n(n-1)}}{1-e^{-\overline\gamma_J}}.
\end{split}
\end{equation}
Therefore, one has for every $i,j=1,\dots, n$,
\begin{equation}
|\sigma_i(t)-\sigma_j(t)|^2\le\ep+ \frac{\mu_1 +\mu_2(N-n)\sqrt{2n(n-1)}}{1-e^{-\overline\gamma_J}}<\ep+ M,
\end{equation}
whenever $t-t_0>T(J,\ep)=\frac{1}{\overline \gamma_J}\ln\left(4\rho^2/\ep\right)>0$, which concludes the proof.
\end{proof}

Also in the case of cluster synchronization, a sharper result can be obtained by considering a stronger assumption than the uniform $L^1_{loc}$-boundedness in the statement of Theorem~\ref{thm:sync_cluster}, i.e.~boundedness in 
$L^\infty$. 

\begin{corollary}
Under the assumptions of {\rm Theorem~\ref{thm:sync_cluster}}, if additionally $|\beta(\cdot)|\in L^\infty$, and also $a_{jk}(\cdot),a_{ik}(\cdot)\in L^\infty$ for all $i,j\in J$ and $k\notin J$, then, for every $\ep>0$ there is $T(\ep)=\frac{1}{\gamma_J}\ln\left(4\rho^2/\ep\right)>0$ such that for all $i,j=1,\dots,N$
\begin{equation}
|\sigma_i(t)-\sigma_j(t)|^2< \ep+\frac{1}{\overline \gamma_J}\Big(\|\beta(\cdot)\|_{L^\infty}+2\rho^2(N-n)\sqrt{2n(n-1)}\max_{\substack{i,j\in J,\\ k\notin J}}\|a_{jk}(\cdot)-a_{ik}(\cdot)\|_{L^\infty}\Big),
\end{equation}
for $t-t_0>T(\ep)$. 
\end{corollary}
\begin{proof}
The result is an easy consequence of Theorem~\ref{thm:sync_cluster} reasoning as for Corollary~\ref{cor:sharp_sync}.
\end{proof}

The next example highlights how the up-to-a-constant synchronization achieved in finite time through Theorem \ref{thm:sync_cluster}, can be used to produce recurrent patterns of cluster synchronization alternated by intervals with no synchrony.

\begin{example}[Clustering in a FitzHugh-Nagumo network] To showcase Theorem~\ref{thm:sync_cluster}, let us consider a network of heterogeneous FitzHugh-Nagumo neurons. The $i$-th neuron has dynamics given by
\begin{equation}
    \begin{split}
        \dot x_i &=c_ix_i - x_i^3 - y_i + I_i\\
        \dot y_i &= \ve(x_i+a_i-b_iy_i),
    \end{split}\qquad\qquad i=1,\ldots,N.
\end{equation}
In this model $x_i$ represents the $i$-th membrane's voltage and $y_i$ the $i$-th recovery variable~\cite{rocsoreanu2012fitzhugh,rauch1978qualitative}. Regarding the parameters, we have that $c_i>0$ modulates the amplitude of oscillations, $I_i$ accounts for the stimulus current, and $\ve$ stands for the difference in timescales between the voltage and the recovery variables. 
The parameters $a_i>0$ and $b_i>0$, together with $I_i$, determine whether the neuron is in excitatory or in refractory mode. For this example we let the neurons be slightly heterogoeneous, and unless otherwise stated, we  assign uniformly at random parameter values according to the following: $c_i\in[0.75,1]$, $a_i\in[-0.3,\,0.3]$, $b\in[0.1,2]$, $I_i\in[0,.01]$ and $\ve=0.05$. We notice that with these parameters, isolated neurons may or may not oscillate. We moreover have similar arguments as in example \ref{ex:vdp} regarding \ref{H1}, \ref{H2}, and \ref{H3}.

We setup the network as follows: an underlying connected, unweighted, and directed graph of $N$ nodes is randomly generated by selecting an adjacency matrix $A=[a_{ij}]_{i,j=1,\ldots,N}$ with $a_{ij}\in\left\{0,1 \right\}$, from a (discrete) uniform distribution. Below we shall specify a time-varying change in the weights of this underlying network, but we emphasize that no new edges are created. Then, two nodes are chosen at random. Let $(x_l,y_l)$ and $(x_k,y_k)$ be such neurons; we set their parameters to $(c_l,I_l,a_l,b_l)=(0.5,0.1,0.3,1.4)$ and $(c_k,I_k,a_k,b_k)=(0.75,0.15,0.3,1.4)$, which ensures that, in isolation, the $l$ and $k$ neurons are oscillating. Furthermore, we call \emph{the neighbors of neuron $l$ (resp. $k$)} the nodes $i$ for which there is a directed edge $a_{il}$ from $l$ to $i$ (resp. $a_{ik}$ from $k$ to $i$). The set of neighbors of neuron $l$ (resp. $k$) is denoted by $\mathcal N_l$ (resp. $\mathcal N_k$).

\begin{figure}
    \centering
    \includegraphics[clip,trim={3cm 11cm 3cm 11cm}]{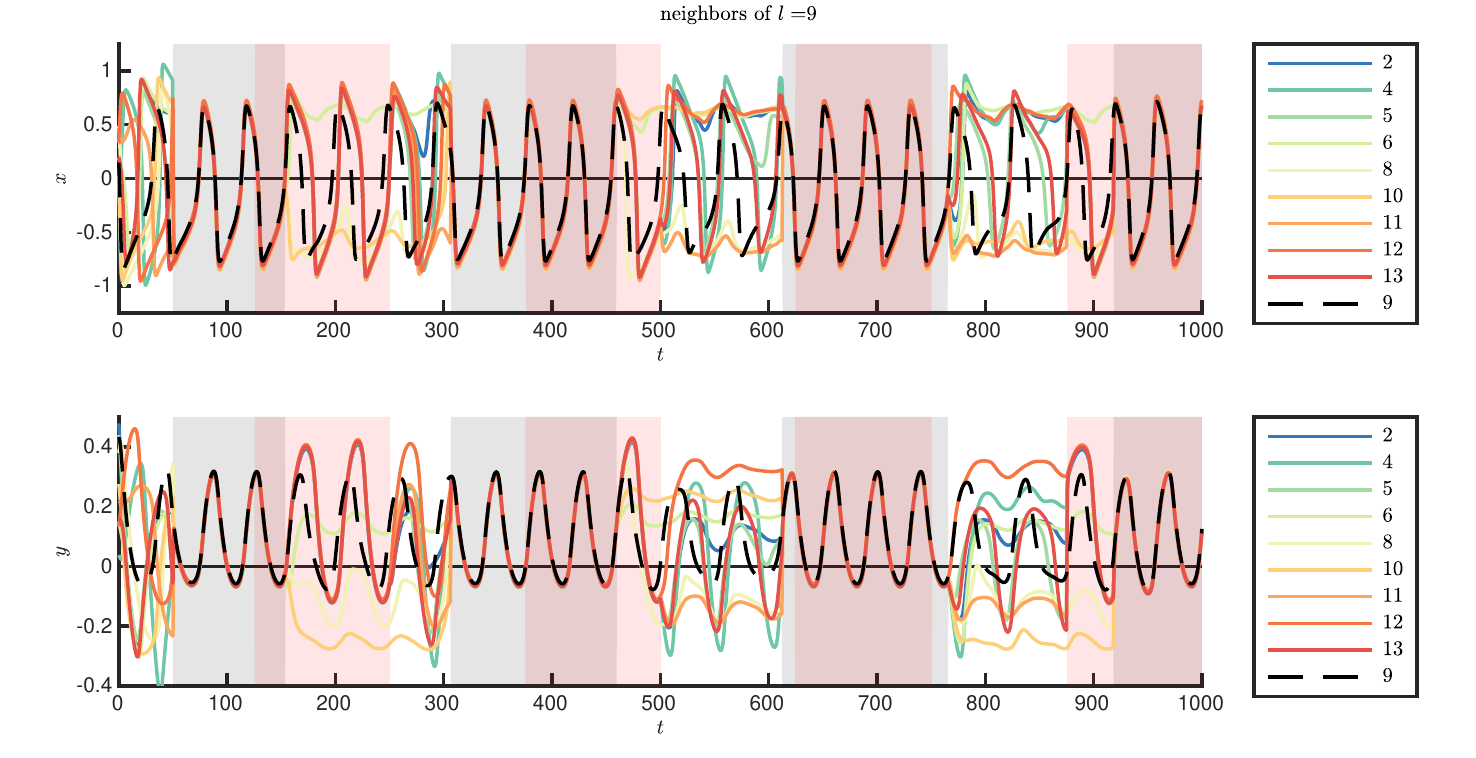}
    \includegraphics[clip,trim={3cm 11cm 3cm 11cm}]{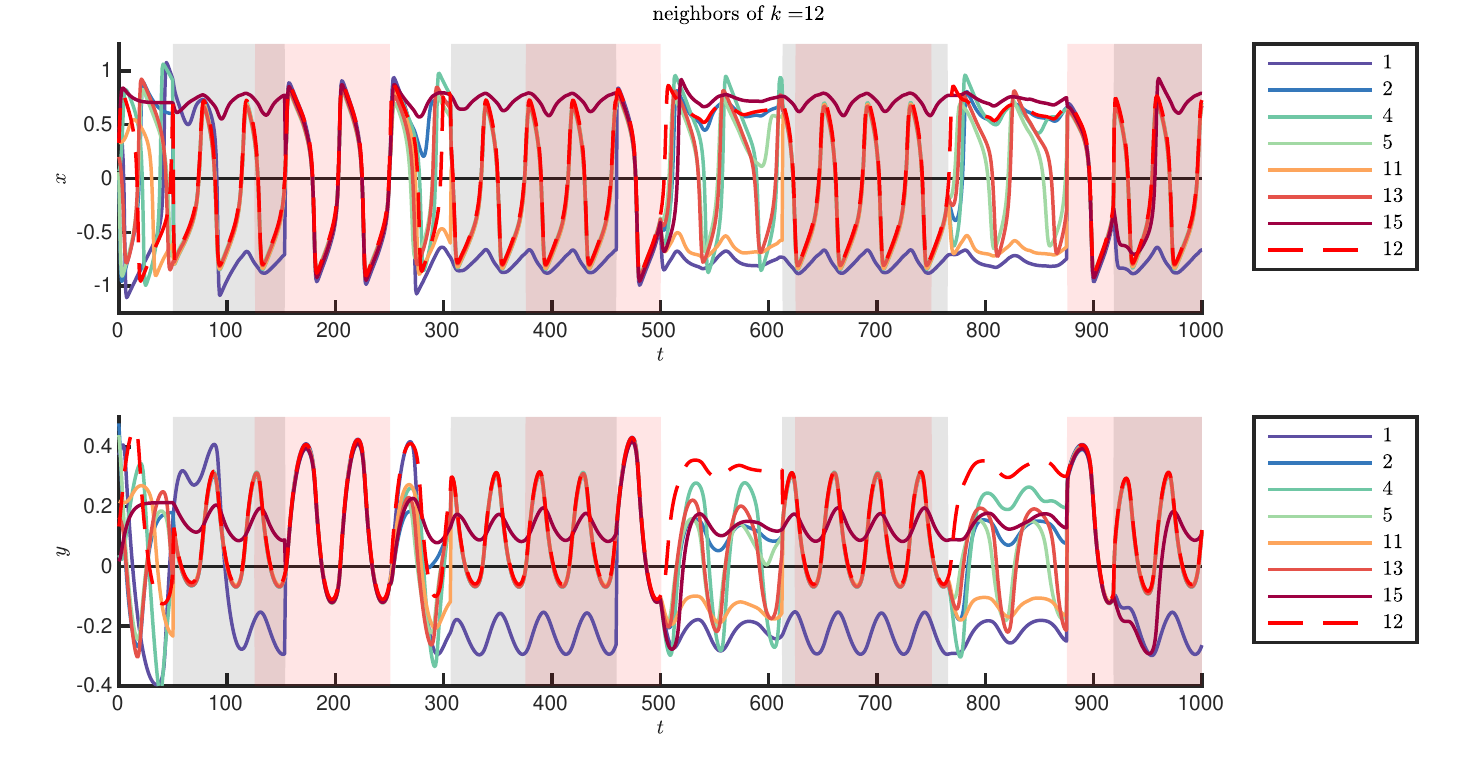}
    \caption{A representative simulation for $N=15$ nodes. As described in the main text, within the black/red shaded time intervals, neighboring neurons synchronize with the dashed black/red neuron. This is because during such time-frames, the conditions of Theorem \ref{thm:sync_cluster} hold. Indeed one can particularly observe that, since $k=12$ is a neighbor of $l=9$, the $k$-th neuron (dashed red) has larger amplitude according to its own parameters during the red time-frames, but synchronizes with the smaller amplitude oscillator $l=9$ (dashed black) during the black time-frames. We further notice that during the overlap of the time-frames, some trajectories also seem to synchronize. This, however, is not characterized in the example, and may very well depend on further connectivity properties. Nevertheless, notice that all neighbors of $k=12$, except for $1$ and $15$, are also neighbors of $l$. During the overlap of the time-frames we hence see a common cluster that does not include the aforementioned neighbors. 
    }
    \label{fig:fn1}
\end{figure}

Next, the (nonzero) weights of the network are set to $a_{ij}=\frac{1}{100}$\footnote{where $A=[a_{ij}]$ is the adjacency matrix of the previously randomly generated graph, and we identify a weighted edge from $j$ to $i$ with $a_{ij}$. If there is no edge from node $j$ to node $i$, then $a_{ij}=0\,\forall t$.}, and the weights of the outgoing edges of the $l$ and $k$ neurons are updated periodically as follows:
\begin{equation}
    a_{il}(t)=\begin{cases}
        \bar a, & \sin(\omega_lt)\geq0\\
        \frac{1}{100}, & \sin(\omega_lt)<0
    \end{cases},\qquad
    a_{ik}(t)=\begin{cases}
        \bar a, & \sin(\omega_kt)\geq0\\
        \frac{1}{100}, & \sin(\omega_kt)<0,
    \end{cases}
\end{equation}
for some positive frequencies $\omega_l$ and $\omega_k$. This example has been setup so that whenever $a_{il}=\bar a$ is sufficiently large and $a_{ik}=\frac{1}{100}$ (or viceversa), Theorem~\ref{thm:sync_cluster} holds for $J$ being the neuron $l$ together with its neighbors $\mathcal N_l$ (or viceversa).

\begin{figure}
    \centering
    \includegraphics[clip,trim={3cm 11cm 3cm 11cm}]{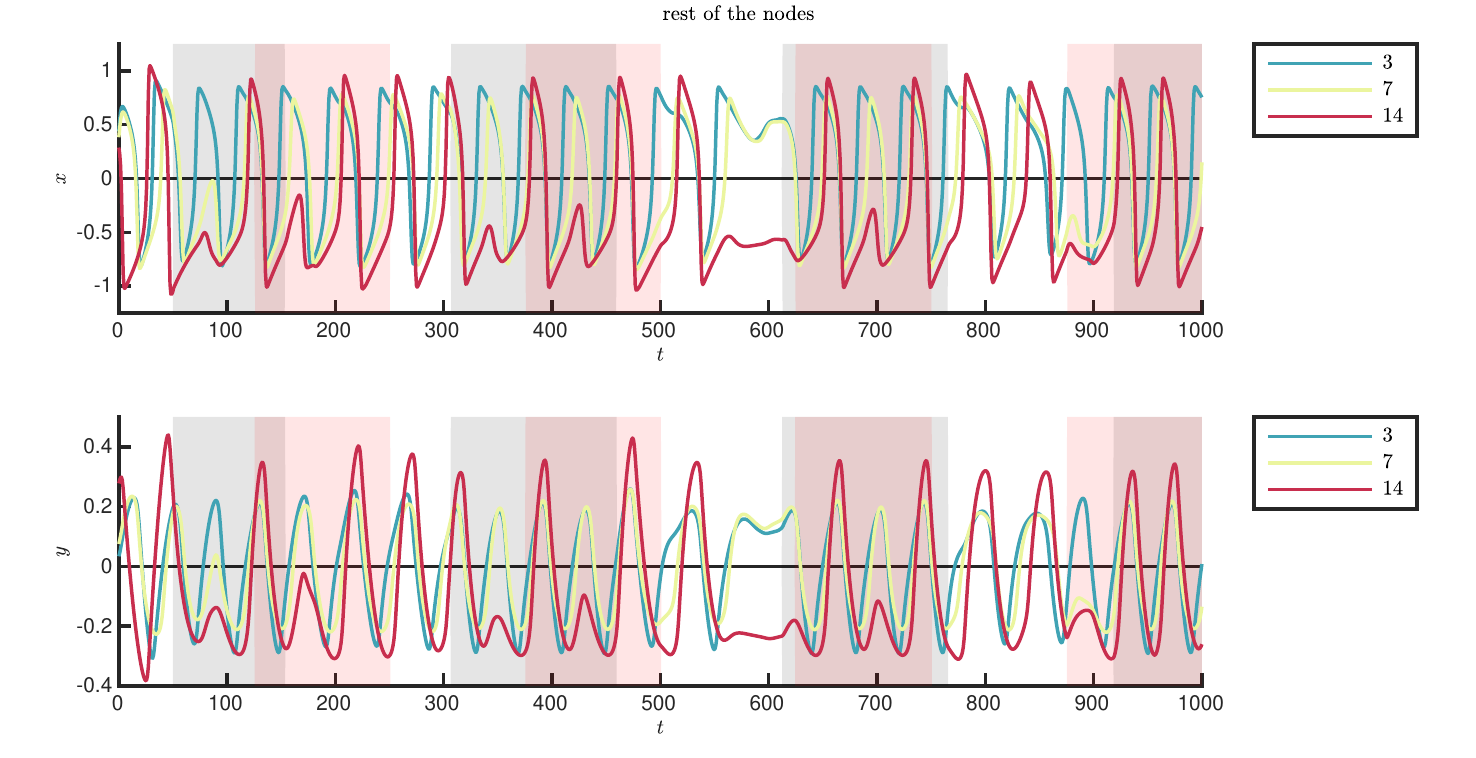}
    \caption{Neurons that are not neighbors of the $l$ nor the $k$ neurons.}
    \label{fig:fn2}
\end{figure}

In figure~\ref{fig:fn1} we present a representative simulation for $N=15$ and $\bar a =3$. The frequencies $\omega_l$ and $\omega_k$ have been chosen so that during the time interval with a black (resp. red) background, the cluster is formed by the $l$ (resp. $k$) neuron, shown as the black (resp. red) dashed curve, and its neighbors. So, indeed notice that along the ``black interval'' the $l$ neuron and its neighbors form a cluster while outside such intervals no synchronization seems to ensue (the same for the red intervals and the $k$ neuron). Outside the black and red intervals, the network is weakly connected with $a_{ij}=\frac{1}{100}$, except for the first $50$ time-units where, for comparison, the network is disconnected. The black and red intervals overlap differently because for this simulation we chose $\omega_l$ and $\omega_k$ incommensurate with each other. In figure~\ref{fig:fn2} we show the nodes that are not neighbors of either the $l$ or the $k$ neurons.
\end{example}

 \section{Persistence of synchronization} \label{sec:persistence}
 One of the interesting features of Theorems~\ref{thm:sync_up_to_constant} and~\ref{thm:sync_cluster} is the property of inherent robustness of the achieved synchronization against perturbation of both, the dynamics of the individual nodes, and of the adjacency matrix. The underlying reason is the roughness of the exponential dichotomy guaranteeing synchronization in Theorems \ref{thm:sync_up_to_constant} and \ref{thm:sync_cluster}. In this sense, the following result reminds in spirit the one in~\cite{pereira2013towards}, although only static networks are therein considered.  In this brief subsection, we aim to make these relations more explicit. For simplicity of notation, we shall treat the case of the entire network, although the same ideas can be applied also to the synchronization of clusters as we briefly highlight in Example \ref{ex:starnetwork}.

\begin{theorem}\label{thm:persistence-of-synchronization}
    Let $f:\R\times \R^M\to\R^M$ be a $\LC$ function and $A:\R\to\R^{N\times N}$ be locally integrable, and consider the network 
    \begin{equation}
     \dot x_i = f(t,x_i) + \sum_{k=1}^N a_{ik}(t)(x_k-x_i),\quad x_i\in\R^M,\, i=1,\dots, N,
    \end{equation}
    Moreover, assume that the assumptions of {\rm Theorem ~\ref{thm:sync_up_to_constant}} are satisfied with $\mu=0$ in~\ref{H2}, and thus sharp synchronization of the entire network is achieved. The following statements are true.
    \begin{itemize}[leftmargin=*,itemsep=2pt]
        \item For every $\delta>0$, such that if $f_i:\R\times \R^M\to\R^M$, for $i=1,\dots, N$, are $\LC$ functions and 
        \begin{equation}   \sup_{\substack{x\in\R^M,\,t\in\R}}|f(t,x)-f_i(t,x)|<\delta,
        \end{equation}
        the perturbed network~\eqref{eq1} synchronizes up to a constant $4\rho\delta\sqrt{N(N-1)/2}/\overline \gamma$, provided that condition~\ref{H2} is still satisfied with $\mu=0$.
        \item If $B:\R\to\R^{N\times N}$, defined by $B(t)=(b_{ij}(t))_{i,j=1,\dots,N}$ is locally integrable and $\sup_{t\in\R_+} |B(t)|<\overline \gamma/4$, then the perturbed network 
        \begin{equation}
         \dot x_i = f(t,x_i) + \sum_{k=1}^N \big(a_{ik}(t)+b_{ik}(t)\big)(x_k-x_i),\quad x_i\in\R^M,\,  i=1,\dots, N,
        \end{equation}
        achieves sharp synchronization.
    \end{itemize}  
\end{theorem}
 \begin{proof}
     In order to prove the first statement, note that for all $t\in\R$,
     \begin{equation}
     \begin{split}
     \langle x-y,&f_i(t,x)-f_j(t,y)\rangle=
     \langle x-y,f_i(t,x)-f(t,x)\rangle+\\ &\qquad\qquad\qquad+\langle x-y,f(t,x)-f(t,y)\rangle+ \langle x-y,f(t,y)-f_j(t,y)\rangle\\
     &\le l^r(t)|x-y|^2+2r\delta,\qquad\text{for all } x,y\in B_r. 
     \end{split}
     \end{equation}
     For the previous chain of inequality we have used Remark~\ref{rmk:f_i=f_j implies_beta=0}, Cauchy-Schwarz inequality and the assumption of $f_i, f_j\in \LC$. Therefore~\ref{H1} and~\ref{H3} hold true with $\beta_{ij}(t)=2r\delta$ for all $t\in\R$ and all $i,j=1,\dots, N$, while~\ref{H2} is satisfied by assumption. Hence, Corollary~\ref{cor:sharp_sync} applies. Noting that $\|\beta(\cdot)\|_{L^\infty}=4\rho\delta\sqrt{N(N-1)/2}$, one has that for all $i,j=1,\dots, N$,
     \begin{equation}
\lim_{t\to\infty}|\sigma_i(t)-\sigma_j(t)|^2<\frac{4\rho\delta\sqrt{N(N-1)/2}}{\overline \gamma}.
\end{equation}

The second statement is a direct consequence of the roughness of the exponential dichotomy and it is obtained applying~\cite[Proposition 4.1]{coppel2006dichotomies}.
\end{proof}

\subsection{An application to time-dependent perturbations of static networks}\label{sec:time_dependent_perturbations}

It is well-known that strongly diffusely coupled, static networks of identical nodes locally synchronize provided that the global coupling overcomes a certain threshold, see e.g.,~\cite{pereira2011stability,pereira2013towards}. The case of nonidentical nodes, under further constraints, has also been considered, see for example~\cite{zhao2010synchronization}. In this section we discuss how our theory relates to this classic result both in terms of synchronization and time-dependent perturbation. Furthermore, we showcase such relations by means of an example on a star network at the end of the section.

Firstly, we show consistency of the two theories: the inequalities of Theorem~\ref{thm:sync_up_to_constant} are always verified by static strongly connected networks  of identical nodes (and as a matter of fact even more general static networks), provided that the global coupling overcomes a certain threshold.

\begin{corollary}\label{cor:sync-static-network}
    Consider a static connected network of identical nodes satisfying \ref{H2} and \ref{H3},
    \begin{equation}
\dot x_i = f(x_i) + c\sum_{k=1}^N a_{ik}(x_k-x_i), x_i\in\R^M,\,  i=1,\dots, N,
    \end{equation}
    with $a_{ij}\in\R$ for all $i,j=1,\dots, N$ and global coupling strength $c>0$. Assume, furthermore, that  $2(a_{ij}+a_{ji})+\sum_{\substack{k=1\\k\neq i,j}}\Big(a_{jk}+a_{ik}-\big|a_{jk}-a_{ik}\big|\Big)>0$ for every $i,j=1,\dots, N$---e.g.,~this is true if the graph is strongly connected and $a_{ij}\ge0$ for all $i,j=1,\dots, N$.
 Then, there is $\overline c>0$ such that for $c> \overline c$ the assumptions of {\rm Theorem~\ref{thm:sync_up_to_constant}} are satisfied and the network achieves synchronization.
\end{corollary}
\begin{proof}
The assumptions on the edges of the network imply that $a_{ij}+a_{ji}
    +\frac{1}{2}\sum_{\substack{k=1\\k\neq i,j}}^N\big(a_{jk}+a_{ik}\big)>0$ for all $i,j=1,\dots, N$. Hence, there is $c_1>0$ such that for any pair $i,j=1,\dots, N$ with $i<j$,
    \begin{equation}
    \delta_{ij}= l^\rho-c\Big(a_{ij}+a_{ji}
    +\frac{1}{2}\sum_{\substack{k=1\\k\neq i,j}}^N\big(a_{jk}+a_{ik}\big)\Big)<0,\qquad\text{for } c>c_1,
    \end{equation}
    where $l^\rho$ is the Lipschitz coefficient for $f$ on $B_\rho$  and the same notation of Theorem~\ref{thm:sync_up_to_constant} has been used. 
    Since $\beta_{ij}$ can be taken equal to zero for identical nodes (see Remark~\ref{rmk:f_i=f_j implies_beta=0}), we also have that for $c>c_1$,
    \begin{equation}
            2|\delta_{ij}|-c\sum_{\substack{k=1\\k\neq i,j}}^N\big|a_{jk}-a_{ik}\big|=2c(a_{ij}+a_{ji})-2l^\rho+c\sum_{\substack{k=1\\ k\neq i,j}}^N\big(a_{jk}+a_{ik}-\big|a_{jk}-a_{ik}\big|\big).
    \end{equation}
    Again, the assumptions on the network topology imply that a $\overline c\ge c_1$ exists such that both inequalities of Theorem~\ref{thm:sync_up_to_constant} are satisfied for $c>\overline c$ and in such a case the system synchronizes.
\end{proof}

Of further interest, is the fact that the synchronization achieved beyond the coupling threshold presented in the previous corollary is robust against small perturbations of the type discussed in Theorem~\ref{thm:persistence-of-synchronization}.

\begin{corollary}\label{cor:perturbed-static-network}
Consider the static network of identical nodes,
    \begin{equation}
        \dot x_i = f(x_i) + c\sum_{k=1}^N a_{ik}(x_k-x_i), x_i\in\R^M,\,  i=1,\dots, N,
    \end{equation}
with coupling coefficients $a_{ik}\ge0$ for all $ i,k=1,\dots, N$ and global coupling strength $c$ greater than the coupling threshold $\overline c$ in {\rm Corollary~\ref{cor:sync-static-network}}. Then, any sufficiently small (possibly time-dependent) perturbation in the sense of {\rm Theorem~\ref{thm:persistence-of-synchronization} } will also produce synchronization.
\end{corollary}
\begin{proof}
    The result is a direct consequence of  Theorem~\ref{thm:persistence-of-synchronization} and Corollary~\ref{cor:sync-static-network}.
\end{proof}

\begin{example}[Synchronization of chaotic oscillators on star networks]\label{ex:starnetwork}
Although dynamics on undirected star networks can be studied, for example, by spectral methods at least in the static case, we can also use Theorem~\ref{thm:sync_up_to_constant} and the Corollaries~\ref{cor:sync-static-network} and~\ref{cor:perturbed-static-network}
to find conditions that lead to synchronization on time-varying directed ones. 

We consider the following setup: the central node (the hub), we call it $x_1$, has directed outgoing edges with weights $a_{i1}=a\in\R$, $i=2,\ldots,N$, with the rest of the network, while the remainder of the nodes (the leaves, or satellites) have corresponding weight $a_{1j}=b\in\R$, $j=2,\ldots,N$; 
every other weight $a_{ij}$ is zero due to the considered graph structure---a schematic representation is shown in  Figure~\ref{fig:star}. We shall hereby assume that all the nodes have identical dynamics (hence $\beta^r_{ij}=0$ for all $r>0$ and all $i,j=1,\dots,N$), and that a global attractor exists and $\alpha^r_{ij}(t)<\alpha\in\R$ for all $i,j=1,\dots,N$ in a sufficiently big ball of radius $r>0$ containing the global attractor. Notice, however, that if $\alpha_{ij}(t)\leq0$ for all $i,j=1,\ldots, N$ and all $t\in\R$, then it suffices to let $\alpha=0$. 

\begin{figure}
    \centering

\begin{tikzpicture}
\def \radius {3cm}

\node[draw, very thick, circle] at (360:0mm) (center) {$x_1$};
\foreach \i [count=\ni from 0] in {N,2,3,4,5,6,7}{
  \node[draw, very thick, circle] at ({90-\ni*45}:\radius) (u\ni) {$x_{\i}$};
  \draw[->] (center) to[bend right=15](u\ni);
  \draw[->] (u\ni) to[bend right=15](center);
}

\draw[line width=3pt, line cap=round, dash pattern=on 0pt off 4\pgflinewidth] (120:.7*\radius) arc[start angle=120, end angle=150, radius=.7*\radius];

\node at (166:0.55*\radius) {$a$};
\node at (184:0.5*\radius) {$b$};
    
    \end{tikzpicture}
\caption{A representation of the star network setup used in Example~\ref{ex:starnetwork}. Node $x_1$ is set as the hub and it has directed outgoing edges with positive weight $a_{i1}=a>0$, $i=2,\ldots,N$. Every other node $x_j$, for $j=2,\ldots,N$ has only one directed outgoing edge towards $x_1$, which has weight $a_{1j}=b<0$.}
    \label{fig:star}
\end{figure}
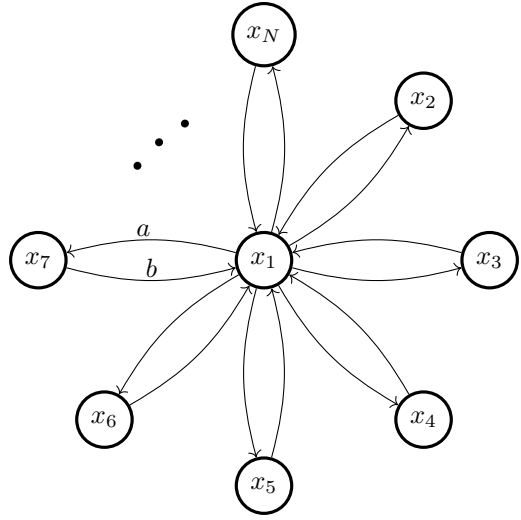

Under the described setting, one finds that, for any pair of satellites, i.e.~$1<i<j$,
\begin{equation}\label{eq:ex_star_sats}
   0<2(a_{ij}+a_{ji})+\sum_{\substack{k=1\\k\neq i,j}}^N\Big(a_{jk}+a_{ik}-\big|a_{jk}-a_{ik}\big|\Big)=2a \quad\Leftrightarrow \quad a>0,
\end{equation}
whereas, for the hub ($i=1$) and any other satellite ($j>1$),
\begin{equation}\label{eq:ex_star}
    \begin{split}          0<2(a_{1j}+a_{j1})&+\sum_{\substack{k=1\\k\neq 1,j}}^N\Big(a_{jk}+a_{1k}-\big|a_{jk}-a_{1k}\big|\Big)=2(b+a)+(N-2)(b-|b|),
    \end{split}
\end{equation}
and the previous inequality is satisfied if and only if
\begin{equation}
\text{either}\quad\circled{A}\,
\begin{cases}
    b<0,\\
    a>-b(N-1),
\end{cases}
\qquad\text{or}\qquad
\circled{B}\,\begin{cases}
    b\ge0,\\
    a\ge-b.
\end{cases}
\end{equation}
Since we assume $a>0$ (due to \eqref{eq:ex_star_sats}), then \circled{B} holds always true, while in case \circled{A} ($b<0$), some care needs to be taken in choosing $a$ depending on the size of the network and the modulus of $b$.  It is also straightforward to extend the previous arguments to the time-varying case $a_{ij}=a_{ij}(t)$, provided that the $a_{ij}$'s are bounded since it would suffice that for almost all $t\geq t_0$,
\begin{equation}
    \begin{split}
        a\le \min_{i=2,\ldots,N}\left\{a_{i1}(t)\right\}\,\qquad\text{and}\qquad
        b\le\min_{j=2,\ldots,N}\left\{ a_{1j}(t) \right\}.
    \end{split}
\end{equation}
In conclusion, on a star network and provided that~\ref{H1}-\ref{H3} hold, no matter the (time-varying, possibly negative) influence of the satellites, the hub can always induce synchronization if $a>0$ is sufficiently big. Moreover the synchronization error can be made as small as desired if a global coupling $c>0$ intervenes---see Corollaries  \ref{cor:global_coupling} and \ref{cor:sync-static-network}.

To verify our arguments, we consider a directed star network where the nodes are Lorenz systems, that is, the internal dynamics of each node is given by
\begin{equation}
    \begin{split}
        \dot x_i &= \sigma(y_i-x_i)\\
        \dot y_i&= x_i(\rho-z_i)-y_i\\
        \dot z_i&=x_iy_i-\beta z_i,
    \end{split}
\end{equation}
where, unless otherwise is stated, the parameters are $\sigma=10$, $\rho=28$, $\beta=\frac{8}{3}$. We focus on the less immediate case of negative edges and set  the influence of the leaves to the hub to $a_{1j}=b=-1$, $j=2,\ldots,N$. Whenever a simulation depends on a random choice of parameters, we always show a representative among dozens of simulations. All simulations have been performed in Matlab using ODE45 with initial conditions randomly chosen near the origin. For the simulations we shall use an error
\begin{equation}\label{eq:l_error}
    \hat e = \sqrt{e_x^2+e_y^2+e_z^2}
\end{equation}
defined by
\begin{equation}
e_\zeta=e_\zeta(t)=\max_{ij}|\zeta_i(t)-\zeta_j(t)|,\qquad \zeta=x,y,z,\; i=1,\ldots,N,
\end{equation}
corresponding to the maximum of pairwise errors at each time $t$.

First, in Figure~\ref{fig:lorenz1} we show simulations for $N=5$ nodes, for three different scenarios, see the description in the caption. The common feature is that following the analysis presented above, synchronization of the network can be achieved. 
\begin{figure}[H]
    \centering
    \includegraphics{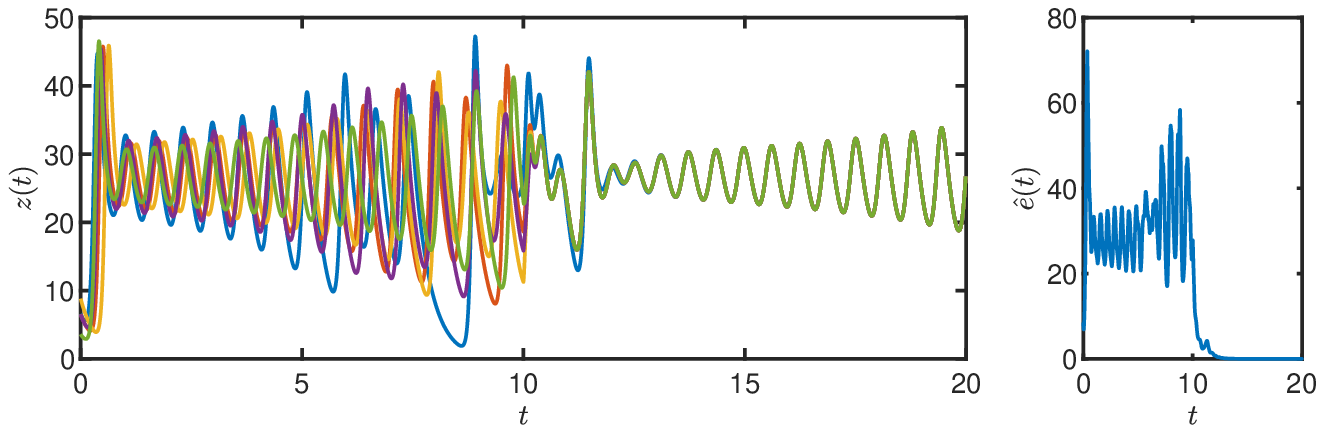}\\
    \includegraphics{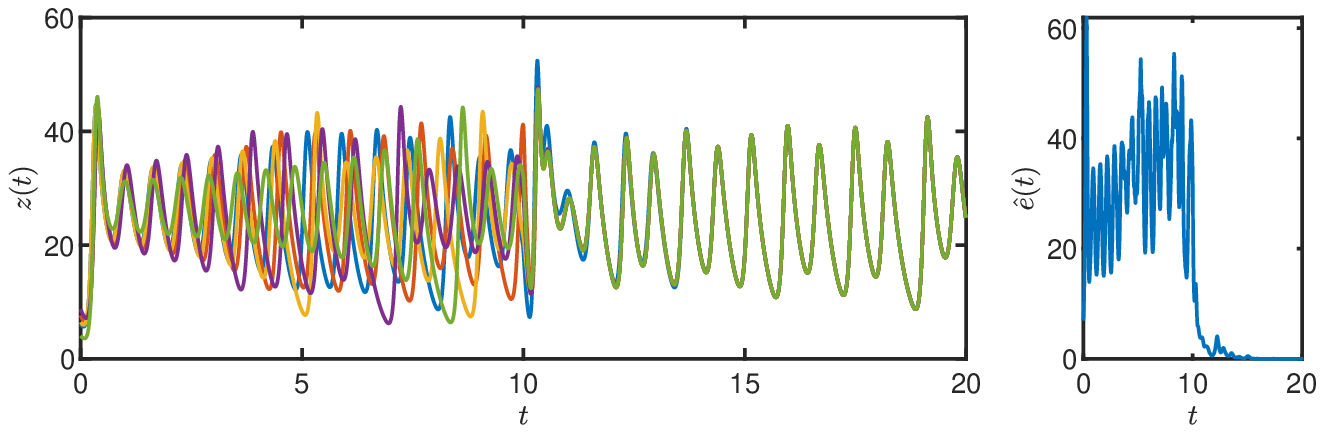}
    \includegraphics{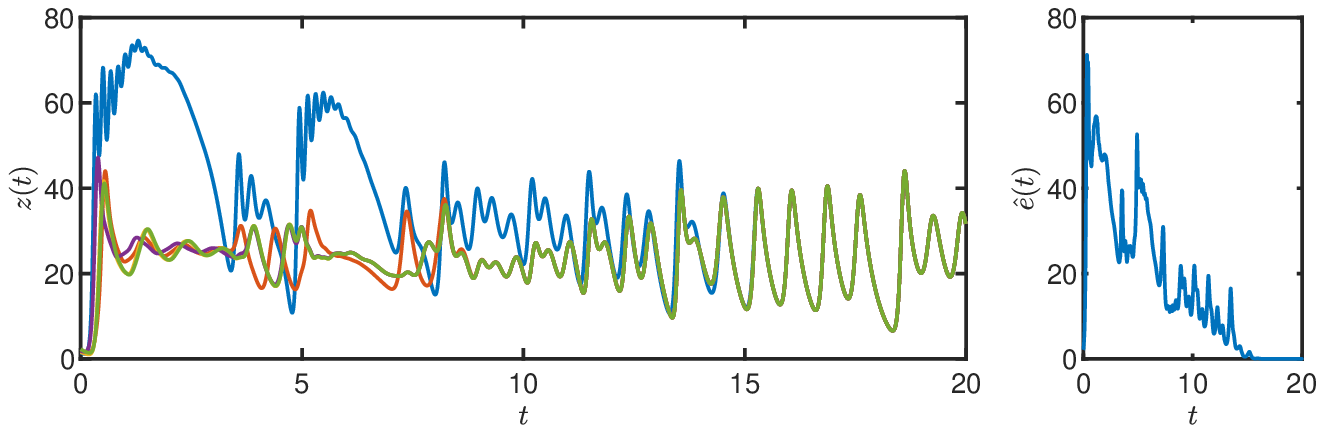}
    \caption{Simulations of three distinct numerical experiments for a star network of $N=5$ Lorenz oscillators. For convenience we show only the time series of the $z$-coordinate on the left and the corresponding error $\hat e$ \eqref{eq:l_error} on the right. In the first plot, we let the network be disconnected for the first 10 time units. At $t=10$ we connect the network with $a_{i1}=a=5$ and $a_{1j}=b=-1$ (see \protect\circled{A} above), and global coupling $c=2$. With these parameters, Corollary \ref{cor:sync-static-network} guarantees synchronization, as verified in the plot for $t\geq10$. To showcase the persistence of synchronization, the second plot shows a similar setting as the first, but with all weights perturbed as $a_{ij}\mapsto a_{ij}(1+\frac{1}{10}\sin(\omega_{ij}t))$ for some randomly chosen frequencies $\omega_{ij}\in(\pi,2\pi)$. Note that for the considered static example $1/10<\min_{i,j=1,\dots,N}\overline \gamma_{ij}/4$ for all $t>10$. Therefore, Corollary \ref{cor:perturbed-static-network} applies. Finally, the third plot shows the case where $a=a(t)=4+3\tanh\left( \frac{1}{5}(t-10)\right)$. Hence this plot shows the transition from \eqref{eq:ex_star} not holding $(a<4)$ to where it does. Notice that since \eqref{eq:ex_star_sats} holds, we observe that first the satellites tend to form a cluster, to later transition to full synchronization. }
    \label{fig:lorenz1}
\end{figure}
Finally, we consider a large network of heterogeneous Lorenz systems. For this we let $N=200$, and all the leaves have randomly picked parameters $\sigma_i\in(\sigma-1,\sigma+1)$, $\rho_i\in(\rho-1,\rho+1)$, $\beta_i\in(\beta-1,\beta+1)$. Similar to the second experiment in Figure~\ref{fig:lorenz1}, we let $a_{1j}=b\left(1+\frac{1}{10}\sin(\omega_{1j}t) \right)$, $b=-1$, and $a_{j1}=-2(N-1)b\left(1+\frac{1}{10}\sin(\omega_{j1}t) \right)$; for some randomly set frequencies $\omega_{ij}\in(\pi,2\pi)$, and $c=1$. For presentation purposes we prefer to show in Figure~\ref{fig:lorenz3} only the corresponding error \eqref{eq:l_error}.

\begin{figure}
    \centering
\includegraphics{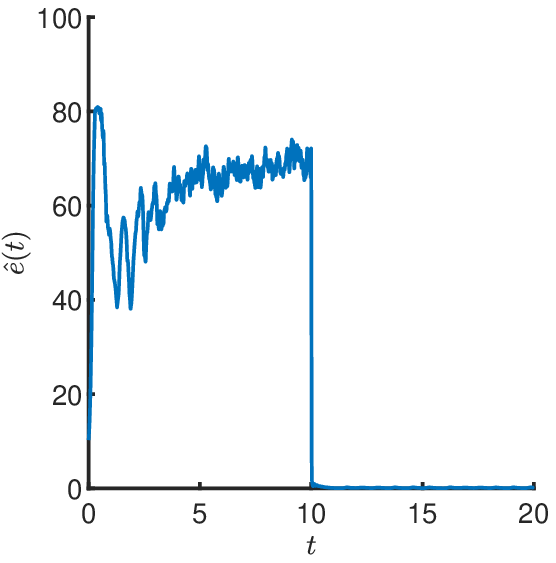}
    \caption{ Synchronization error~\eqref{eq:l_error} for a temporal star network with $N=200$ heterogeneous Lorenz oscillators. In this figure, for the first $10$ time units, the network is disconnected. Afterwards, for $t\geq10$ the network is connected as described above so that Theorem~\ref{thm:sync_up_to_constant} holds, leading to synchronization.}
    \label{fig:lorenz3}
\end{figure}

\end{example}

\section{Existence of local attractors}\label{sec:suff_cond_attr} 

In this final section, we provide a sufficient condition for the existence of an attracting trajectory for each equation of~\eqref{eq1} so that~\ref{H2} is satisfied. The treatment largely owes to the theory developed by Caraballo et al.~\cite{caraballo2008synchronization}. A substantial generalization intervenes in the sense that now all the considered inequalities involve locally integrable functions in place of constants and more than two coupled systems are considered. This is in line with the rest of the treatment in our work. Our fundamental inspiration for this type of generalization is the work by Longo et al.~\cite{longo2019weak}.

\subsection{The uncoupled problem}

Given $f\in\LC$, we shall consider the following assumptions,
\begin{enumerate}[label=\upshape(\textbf{SL1}),leftmargin=32pt,itemsep=2pt]
\item\label{SL}
(\emph{local one-sided Lipschitz continuity}) there exists a nonempty forward invariant uniformly bounded nonautonomous set $\U\subset\R\times\R^M$ and a function $l\in L^1_{loc}$ such that for almost every $t\in\R$, $B_r\subset\U_t$   and,
\begin{equation}\label{eq:SL}
\qquad2\langle x_1-x_2,f(t,x_1)-f(t,x_2)\rangle\le l(t)|x_1-x_2|^2,\ \ \text{for all } x_1,x_2\in \U_t.
\end{equation}
\end{enumerate}
\begin{enumerate}[label=\upshape(\textbf{SL2}),leftmargin=32pt,itemsep=2pt]
\item\label{SL2} (\emph{dissipativity}) 
given~\ref{SL}, there are constants $K\ge1$ and $\gamma>0$ such that 
\begin{equation}\exp\left(\int_{t_0}^tl(s)\,ds\right)\le K e^{-\gamma(t-t_0)},\quad\text{for any }t_0\le t.\end{equation}
\end{enumerate}



For practical reasons, we shall assume that the origin, denoted by the vector $\mathbf{0}\in\R^M$, belongs to $\U_t$ for almost every $t\in\R$. 
This hypothesis is not restrictive: let us assume that there is a set $S\subset \R$ with positive measure for which $\mathbf{0}\notin \U_t$ for all $t\in S$.  
Since $\U$ is nonempty, uniformly bounded and forward invariant there is at least one entire solution $\zeta(t)$ whose graph is contained in $\U$.
Then, the time dependent change of variables $(t,y)=(t,x-\zeta(t))$ returns a new vector field $\widetilde f$ and a new forward invariant uniformly bounded nonautonomous set $\widetilde \U$ for which $\mathbf{0}\in \widetilde \U_t$ for almost all $t\in\R$.
This fact will be important in the proof of some of the following results.\par\smallskip

Next, we show that condition~\ref{SL} can be extended to pairs of continuous functions and the inequality~\eqref{eq:SL} keeps holding almost everywhere.
\begin{proposition}\label{prop:one-sided-dissipative-for-functions}
Let $f\in\LC$ and assume~\ref{SL} holds. If $I\subset \R$ is an interval and $\phi,\psi\in C(I,\R^M)$ with $\phi(t),\psi(t)\in\U_t$ for almost every $t\in I$, then
\begin{equation}
2\langle \phi(t)-\psi(t),f\big(t,\phi(t)\big)-f\big(t,\psi(t)\big)\rangle\le l(t)|\phi(t)-\psi(t)|^2
\end{equation}
for almost every $t\in I$.
\end{proposition}
\begin{proof}
A proof of this statement can be obtained reasoning as for Lemma~\ref{lem:two_nodes_assumptions_from_points_to_functions}.
\end{proof}

Note that~\ref{SL} guarantees that each pair of solutions $x(t),y(t)$ of the uncoupled problem $\dot x =f(t,x)$ with initial conditions in $\U$ will converge in forward time. Indeed, one immediately has that 
\begin{equation}
\frac{d}{dt}|x(t)-y(t)|^2\le l(t)|x(t)-y(t)|^2.
\end{equation}
Hence, using~\ref{SL2}we have that for any $t,t_0\in\R$, with $t\ge t_0$, and $x_0,y_0\in\U_{t_0}$,
\begin{equation}
|x(t)-y(t)|^2\le K e^{-\gamma(t-t_0)}|x_0-y_0|^2.
\end{equation}

In order to understand towards what these solutions are converging, we firstly have to show that~\ref{SL} implies a more standard dissipative condition.

\begin{proposition}\label{prop:dissipativity_condition}
Let $f\in\LC$ and assume~\ref{SL} holds. Then, $f$ is asymptotically dissipative, that is, for almost every $t\in\R$, it satisfies 
\begin{equation}\label{eq:dissipativity_condition}
2\langle x, f(t,x)\rangle\le \alpha(t)|x|^2 + \beta(t), \quad \text{for all } x\in\U_t,
\end{equation}
where $\alpha(\cdot),\beta(\cdot)\in L^1_{loc}$ and there is $0<\overline \gamma<\gamma$ such that
\begin{equation}\label{eq:one_node_dissipativity_cond_2}
\exp\left(\int_{t_0}^t\alpha(s)\,ds\right)\le K e^{-\overline\gamma(t-t_0)},\quad\text{for any }t_0\le t.
\end{equation}
\end{proposition}
\begin{proof}
Firstly notice that
\begin{equation}
\begin{split}
\langle x, y\rangle&=\frac{1}{2}(|x+y|^2-|x|^2-|y|^2)\le \frac{1}{2}(|x|^2+|y|^2+2|x||y|-|x|^2-|y|^2)\\
&=\frac{1}{2}\big(|x|^2+|y|^2-(|x|-|y|)^2\big)\le |x|^2+|y|^2. 
\end{split}
\end{equation}
Fix $0<\ep<\gamma/2$ and $x\in\U_t$. Since~\ref{SL} holds, we have that
\begin{equation}
\begin{split}
  2\langle x,f(t,x)\rangle
  &\le 2\langle x,f(t,\mathbf{0})\rangle +l(t)|x|^2 \le 2\ep|x|^2+\frac{2}{\ep}| f(t,\mathbf{0})|^2  +l(t)|x|^2\\
  &=\big(2\ep+l(t)\big)|x|^2+\frac{2}{\ep}| f(t,\mathbf{0})|^2.
\end{split}
\end{equation}
Denoted $\alpha(t):=2\ep+l(t)$ and $\beta(t):=2/\ep| f(t,\mathbf{0})|^2$, note that we have 
\begin{equation}
\exp\left(\int_{t_0}^t\alpha(s)\,ds\right)= e^{2\ep(t-t_0)}\exp\left(\int_{t_0}^tl(s)\,ds\right)\le K e^{(2\ep-\gamma)(t-t_0)},\quad\text{for any }t_0\le t,
\end{equation}
which concludes the proof with $\overline \gamma=\gamma-2\ep$.
\end{proof}




It is now possible to prove that any uncoupled system $\dot x= f(t,x)$ satisfying~\ref{SL} admits a bounded local pullback and forward attracting trajectory, provided that the set, 
\begin{equation}
\{\beta_t(\cdot)\}_{t\in\R}=\{| f(t+\cdot,\mathbf{0})|^2\mid t\in\R\},
\end{equation}
is $L^1_{loc}$-bounded.

\begin{proposition}\label{prop:pullback_forward_uncoupled}
Let $f\in\LC$, satisfying~\ref{SL} and~\ref{SL2} Moreover, let $\alpha,\beta\in L^1_{loc}$ be the functions provided by {\rm Proposition~\ref{prop:dissipativity_condition}}. If $\{\beta_t(\cdot)\}_{t\in\R}$ is $L^1_{loc}$-bounded, 
then $\dot x =f(t,x)$ has a bounded local pullback attractor made of a single globally defined trajectory which is also forward attracting.
\end{proposition}
\begin{proof}
Thanks to Proposition~\ref{prop:dissipativity_condition} we have that for any $t\in\R$, $s>0$ and $x_0\in\U_{t-s}$,
\begin{equation}
|x(t,t-s,x_0)|^2\le K e^{-\overline \gamma s}|x_0|^2+K\int_{t-s}^t\beta(u)e^{-\overline \gamma (t-u)}\, du.
\end{equation}
Reasoning as for~\eqref{eq:VOC_nonhomogeneous_integral_estimate}, and taking $x_0\in\U_{t-s}$, one obtains that 
\begin{equation}
|x(t,t-s,x_0)|^2\le K e^{-\overline \gamma s}r^2+\frac{K\mu }{1-e^{-\overline\gamma}},
\end{equation}
where $\mu:=\sup_{\tau\in\R}\int_\tau^{\tau+1}|\beta(s)|\, ds<\infty$, and $\U_s\subset B_r$ for some $r>0$ by assumption. Therefore, the ball of radius $\ep+K\mu/(1-e^{-\overline\gamma})$ for $\ep>0$ chosen as small as desired, pullback absorbs in finite time the fiber $\U_{t-s}$. This fact implies that there is a universe of attraction contained in $B_r$ that is pullback absorbed into the ball of radius $\ep+K\mu/(1-e^{-\overline\gamma})$. Therefore, a unique bounded  pullback attractor exists~\cite[Theorem 7.2]{kloeden2020introduction}.
On the other hand, Proposition~\ref{prop:one-sided-dissipative-for-functions} and the remark thereafter imply forward convergence of all the trajectories of $\dot x =f(t,x)$. Therefore, one immediately has that all the sections of the pullback attractor are singleton sets. This fact concludes the proof.
\end{proof}

If instead of~\ref{SL} and~\ref{SL2} the weaker conditions~\eqref{eq:dissipativity_condition} and~\eqref{eq:one_node_dissipativity_cond_2} are assumed, a weaker result of uniform ultimate boundedness of solution is obtained.

\begin{corollary}
Let $f\in\LC$, satisfying~\eqref{eq:dissipativity_condition} and~\eqref{eq:one_node_dissipativity_cond_2}. If $\{\beta_t(\cdot)\}_{t\in\R}$ is $L^1_{loc}$-bounded, 
then $\dot x =f(t,x)$ has a bounded local pullback attractor and the rest of trajectores starting in $\U$ are uniformly ultimately bounded.
\end{corollary}
\begin{proof}
The existence of a unique bounded local pullback attractor is achieved following the same initial steps in the proof of Proposition~\ref{prop:pullback_forward_uncoupled}. Now consider $t>t_0$. Thanks to Proposition~\ref{prop:dissipativity_condition} and taking $x_0\in\U_{t_0}$, one obtains that 
\begin{equation}
|x(t,t_0,x_0)|^2\le K e^{-\overline \gamma (t-t_0)}|x_0|^2+K\int_{t_0}^t\beta(u)e^{-\overline \gamma (t-u)}\, du.
\end{equation}
Hence, reasoning as for~\eqref{eq:VOC_nonhomogeneous_integral_estimate}, and taking $x_0\in\U_{t_0}$, one obtains that 
\begin{equation}
|x(t,t_0,x_0)|^2\le K e^{-\overline \gamma (t-t_0)}r^2+\frac{K\mu }{1-e^{-\overline\gamma}},
\end{equation}
and for any $\ep>0$ there is a time $T(\ep)>0$ such that if $t-t_0>T(\ep)$ then 
\begin{equation}
|x(t,t_0,x_0)|^2\le \ep+\frac{K\mu }{1-e^{-\overline\gamma}},
\end{equation}
which completes the proof.
\end{proof}

\subsection{The coupled problem}
Now we turn our attention to the coupled problem~\eqref{eq1}. 
The natural question is if, despite the coupling, each node still has a pullback and forward attracting trajectory. An important role in our argument shall be played by the $L^1_{loc}$ linear operator defined for any $t\in\R$ by 
\begin{equation}
\R^M\ni u\mapsto(2A(t)-Id_{N\times N}L(t))u
\end{equation}
where $L(t)=(l^\rho_i(t))^\top_{i=1,\dots, N}$, and for each $i=1,\dots, N$, $l_i(\cdot)$ is the $L^1_{loc}$ function provided by~\ref{SL}. As for the uncoupled case, we firstly deal with the forward attractivity. 
\begin{lemma}\label{lem:linear_control_difference_trajectories_one_node}
    Consider $f_i\in\LC$ for $i=1,\dots, N$ satisfying~\ref{SL}, and the networked system~\eqref{eq1} where, for almost every $t\in\R$, $a_{ij}(t)\ge0$ for all $i,j=1,\dots,N$, with $i\neq j$ and $a_{ii}(t)=0$ for all $i=1,\dots, N$. Then, considered any pair of absolutely continuous functions $y(t)$ and $z(t)$ respectively solving~\eqref{eq1} with initial data $\overline y,\overline z\in\U^i_{t_0}$ for some $t_0\in\R$, it holds that 
    \begin{equation}
    \big(|y_i(t)-z_i(t)|^2\big)^\top_{i=1,\dots, N}\le U(t,t_0)(|\overline y_i-\overline z_i|^2)^\top_{i=1,\dots, N}
    \end{equation}
    where $U(t,t_0)$ is the principal matrix solution  of $\dot u=(2A(t)-Id_{N\times N}L(t))u$ at $t_0\in\R$.
\end{lemma}
\begin{proof}
Note that for every $i=1,\dots, N$ and using the Cauchy-Schwarz inequality,
\begin{equation}
\begin{split}
   \frac{d}{dt} |y_i(t)&-z_i(t)|^2
   =2\big\langle y_i(t)-z_i(t),f_i\big(t,y_i(t)\big)-f_i\big(t,z_i(t)\big)\big\rangle+\\
   &\qquad\qquad \, +2\big\langle y_i(t)-z_i(t),\sum_{k=1}^Na_{ik}(t)\big(y_k(t)-y_i(t)-z_k(t)+z_i(t)\big)\big\rangle\\
   &\le l_i(t)|y_i(t)-z_i(t)|^2 -2\sum_{k=1}^Na_{ik}(t)\big\langle y_i(t)-z_i(t),y_i(t)-z_i(t)\big\rangle+\\
   &\qquad\qquad \,+2\sum_{k=1}^Na_{ik}(t)\big\langle y_i(t)-z_i(t),y_k(t)-z_k(t)\big\rangle\\
   &\le l_i(t)|y_i(t)-z_i(t)|^2-2\sum_{k=1}^Na_{ik}(t)|y_i(t)-z_i(t)|^2+\\
   &\qquad\qquad \,+2\sum_{k=1}^Na_{ik}(t)\big[|y_i(t)-z_i(t)|^2+|y_k(t)-z_k(t)|^2\big]\\
   &=l_i(t)|y_i(t)-z_i(t)|^2+2\sum_{k=1}^Na_{ik}(t)|y_k(t)-z_k(t)|^2.
\end{split}
\end{equation}
Therefore, we have that for $t_0\le t$ where it is well-defined, the vector $(|y_i(t)-z_i(t)|^2)^\top_{i=1,\dots, N}$ is an \emph{under-function} with respect to the initial value problem $u'=(2A(t)-Id_{N\times N}L(t))u$, $u(s)=|y(s)-z(s)|^2$, where $L(t)=(l^\rho_i(t))^\top_{i=1,\dots, N}$.  
 Therefore, we can reason as for the proof of Theorem~\ref{thm:sync_up_to_constant} to obtain that for all $t>t_0$,
\begin{equation}
\big(|y_i(t)-z_i(t)|^2\big)^\top_{i=1,\dots, N}\le U(t,t_0)(|y_i(t_0)-z_i(t_0)|^2)^\top_{i=1,\dots, N}
\end{equation}
where $U(t,s)$ is the principal matrix solution of $\dot u=(2A(t)-Id_{N\times N}L(t))u$ at $s\in\R$. Incidentally, this shows that $y(t)$ and $z(t)$ are defined for all $t\ge t_0$.
\end{proof}

Lemma~\ref{lem:linear_control_difference_trajectories_one_node} allows to immediately obtain a sufficient condition for forward convergence of the trajectories of each node. If the linear differential problem $\dot u=2(A(t)-Id_{N\times N}L(t))u$ has dichotomy spectrum contained in $(-\infty,0)$, then each node has a bounded forward attracting trajectory.
\begin{theorem}\label{thm:existence_attractors_coupled}
Consider $f_i\in\LC$ for $i=1,\dots, N$ and assume that they all satisfy~\ref{SL}, each within a forward invariant uniformly bounded nonautonomous set $\U^i\subset\R\times\R^M$. Moreover, assume that $\dot u=2(A(t)-Id_{N\times N}L(t))u$ has dichotomy spectrum contained in $(-\infty,0)$.
Then, there is an absolutely continuous function $\sigma:\R\to\R^{M\times N}$, $t\mapsto \sigma(t)=\big(\sigma_i(t)\big)_{i=1,\dots, N}$ that solves~\eqref{eq1} and such that, for every $i=1,\dots, N$, $\sigma_i(\cdot)$ is pullback attracting for 
\begin{equation}
\dot x_i = f_i(t,x_i) + \sum_{j=1}^N a_{ij}(t)(\sigma_j(t)-x_i),
\end{equation} and if $y(t)=\big(y_i(t)\big)_{i=1,\dots, N}$ solves~\eqref{eq1} with $y_i(t_0)\in\U^i_{t_0}$ for some $t_0\in\R$, then for every $i=1,\dots, N$,
\begin{equation}
\lim_{t\to\infty}|y_i(t)-\sigma_i(t)|=0.
\end{equation}
\end{theorem}
\begin{proof}
Since $\dot u=2(A(t)-Id_{N\times N}L(t))u$ has dichotomy spectrum contained in $(-\infty,0)$, we have that there are $K\ge 1$ and $\gamma>0$, such that 
\begin{equation}
|U(t,t_0)|\le Ke^{-\gamma(t-t_0)}\quad \text{for } t_0\le t,
\end{equation}
where $U(t,t_0)$ is the principal matrix solution  of $\dot u=(2A(t)-Id_{N\times N}L(t))u$ at $t_0\in\R$. 
Therefore, thanks to Lemma~\ref{lem:linear_control_difference_trajectories_one_node}, we have that for any $t_0\in\R$ and  $\overline y=(\overline y_i)_{i=1,\dots, N},\overline z=(\overline z_i)_{i=1,\dots, N}\in\R^M$ with $\overline y_i,\overline z_i\in\U^i_{t_0}$ for all $i=1,\dots,N$, the solutions $y(t,t_0,\overline y)=\big(y_i(t,t_0,\overline y_i)\big)_{i=1,\dots, N}$ and $z(t,t_0,\overline z)=\big(z_i(t,t_0,\overline z_i)\big)_{i=1,\dots, N}$ of~\eqref{eq1} with initial data $y(t_0)=\overline y$ and $z(t_0)=\overline z$, respectively, are defined for all $t\ge t_0$ and in particular,
\begin{equation}
\big(|y_i(t,t_0,\overline y_i)-z_i(t,t_0,\overline z_i)|^2\big)^\top_{i=1,\dots, N}\le Ke^{-\gamma (t-t_0)}(|\overline y_i-\overline z_i|^2)^\top_{i=1,\dots, N},
\end{equation}
for any $t>t_0$.
Particularly, since each nonautonomous set $\U^i$, $i=1,\dots,N$, is  uniformly bounded, then there is $r>0$ such that $\U^i_t\subset B_{\sqrt{r}/2}$ for all $t\in\R$ and $i=1,\dots,N$. Then, fixed $\ep>0$, and considered $s\ge \log(rK/\ep)/\gamma=:T(r,\ep)$, one has that for any $t\in\R$,
\begin{equation}
|y_i(t,t-s,\overline y_i)-z_i(t,t-s,\overline z_i)|^2\le \ep\quad \text{for $ s\ge T(r,\ep)$}.
\end{equation}
Consequently, the function $\zeta(t)=(0,\dots,0)^\top\in\R^M$ pullback and forward attracts all the trajectories $\big(|y_i(t,t-s,\overline y_i)-z_i(t,t-s,\overline z_i)|^2\big)^\top_{i=1,\dots, N}$ of the dynamical system induced by Proposition~\ref{prop:dyn_sys_from_difference}. In particular, this implies that there is a globally defined and absolutely continuous function $\sigma:\R\to\R^{M\times N}$ that solves~\eqref{eq1} and satisfies the thesis.
\end{proof}

As we have recalled in Remark~\ref{rmk:ED}, row-dominance is a sufficient condition for the existence of an exponential dichotomy. Therefore, the stronger dissipativity condition than~\ref{SL} guarantees the persistence of an attractor also for the coupled network as we show in the next corollary. Particularly, we shall henceforth assume that 
\begin{enumerate}[label=\upshape(\textbf{SL1*}),leftmargin=37pt,itemsep=2pt]
\item\label{SL1*}
for every $i=1,\dots,N$ there exists a function $l_i\in L^1_{loc}$ such that,~\eqref{eq:SL} is verified and
\begin{equation*}\label{eq:SL1}
\sup_{\substack{t\in\R,\\ i=1,\dots,N}}l_i(t)<0\quad\text{and}\quad \gamma=\sup_{\substack{t\in\R,\\ i=1,\dots,N}}\bigg\{|l_i(t)|-2\sum_{\substack{k=1\\k\neq i}}^Na_{ik}(t)\bigg\}>0.
\end{equation*}
\end{enumerate}

\begin{corollary}\label{cor:node_i_dynamics_controlled_by_linear_system}
Consider the networked system~\eqref{eq1} where, for almost every $t\in\R$, $a_{ij}(t)\ge0$ for all $i,j=1,\dots,N$, with $i\neq j$ and $a_{ii}(t)=0$ for all $i=1,\dots, N$, and assume further that $f_i\in\LC$  satisfies~\ref{SL1*} for all $i=1,\dots, N$.
Then, the thesis of Theorem~\ref{thm:existence_attractors_coupled} holds true.
\end{corollary}
\begin{proof}

The assumption~\ref{SL1*} implies that the linear problem $\dot u=(2A(t)-Id_{N\times N}L(t))u$ is row dominant and that its dichotomy spectrum is contained in $(-\infty,0)$. Therefore the thesis is a direct consequence of Theorem~\ref{thm:existence_attractors_coupled}.
\end{proof}

\section{Conclusion and Discussion}

In this paper we have considered general linearly and diffusely coupled networks. First, the nodes themselves, with internal dynamics $f_i(t,x_i)$ are of the Lipschitz Carath\'eodoy class, meaning that, in particular, continuity in time is not even required. 
Second, the interaction between nodes is also quite general. In essence we consider temporal interconnections, described by $a_{ij}(t)\in\R$, that are locally integrable. This includes (but it is not restricted to), for example, piece-wise continuous changes in the topology of the network. Our main goal has been to provide quantitative conditions under which synchronization of the nodes is achieved. One of the main highlights of the results we provide is that they mainly depend on the network structure, i.e., on the $a_{ij}$'s. This offers important advantages compared to some other criteria that can, for example, depend on spectral properties, and at the same time allows for a control approach to synchronization. Among the results we have presented we emphasize the synchronization (up to a constant) and the synchronization of clusters, of temporal networks. We have presented some further results, for example for networks with global couplings, where the synchronization error can be made arbitrarily small. Another striking feature of the developed sufficient results of synchronization is their robustness against perturbation due to the roughness of the exponential dichotomy on which they are based. 

A limitation of the presented theory is that we require that all the components of (higher dimensional) nodes are connected in a uniform way---no inner-coupling matrix is considered; although it might be possible to also include this case by using weighted inner products for our proofs in analogy with the idea of the Mahalanobis distance, which is often employed in statistics in such an inner coupling between variables. In summary, as we exemplified through the paper, several quite general and practically relevant situations for synchronization can be covered by our theory. As future work, one could attempt to improve the presented results by allowing that only some of the components of the nodes are to be connected. This is reminiscent of under-actuated control laws. Another interesting extension would include nonlinear interaction functions, as well as fully adaptive networks where the network topology depends on the node dynamics and vice versa node dynamics depends on the topology~\cite{Berneretal}. 

Finally, we notice that since a topological theory for the construction of continuous flows for Carathéodory differential equations exists for ordinary ~\cite{artstein1977topological,longo2018topologies,longo2017topologies, longo2019weak}, delay~\cite{longo2019topologies,longo2021monotone} and parabolic~\cite{longo2022topologies} differential problems, not only the ideas here exposed could be extended to these contexts but also additional results of propagation of synchronization in the hull could be explored.

\section*{Acknowledgements} 
CK was partly supported by a Lichtenberg Professorship of the VolkswagenStiftung. CK was also supported by a DFG Sachbeihilfe grant 444753754. IPL was partly supported  by UKRI under the grant agreement EP/X027651/1, by the European Union’s Horizon 2020 research and innovation programme under the Marie Skłodowska-Curie grant agreement No 754462, by MICIIN/FEDER project PID2021-125446NB-I00, by TUM International Graduate School of Science and Engineering (IGSSE) and by the University of Valladolid under project PIP-TCESC-2020. CK and IPL also acknowledge partial support of the EU within the TiPES project funded by the European Unions Horizon 2020 research and innovation programme under grant agreement No. 820970.\par\smallskip

Furthermore, we thank two anonymous reviewers whose suggestions have been implemented in the current version of the paper.

\bibliographystyle{siam}
\bibliography{references}

\begin{thebibliography}{10}

\bibitem{amritkar2006synchronized}
{\sc R.~Amritkar and C.-K. Hu}, {\em Synchronized state of coupled dynamics on
  time-varying networks}, Chaos: An Interdisciplinary Journal of Nonlinear
  Science, 16 (2006), p.~015117.

\bibitem{artstein1977topological}
{\sc Z.~Artstein}, {\em Topological dynamics of an ordinary differential
  equation}, Journal of Differential Equations, 23 (1977), pp.~216--223.

\bibitem{bassett2013task}
{\sc D.~S. Bassett, N.~F. Wymbs, M.~P. Rombach, M.~A. Porter, P.~J. Mucha, and
  S.~T. Grafton}, {\em Task-based core-periphery organization of human brain
  dynamics}, PLoS computational biology, 9 (2013), p.~e1003171.

\bibitem{becchetti2020consensus}
{\sc L.~Becchetti, A.~Clementi, and E.~Natale}, {\em Consensus dynamics: An
  overview}, ACM SIGACT News, 51 (2020), pp.~58--104.

\bibitem{belykh2004blinking}
{\sc I.~V. Belykh, V.~N. Belykh, and M.~Hasler}, {\em Blinking model and
  synchronization in small-world networks with a time-varying coupling},
  Physica D: Nonlinear Phenomena, 195 (2004), pp.~188--206.

\bibitem{belykh2004connection}
{\sc V.~N. Belykh, I.~V. Belykh, and M.~Hasler}, {\em Connection graph
  stability method for synchronized coupled chaotic systems}, Physica D:
  nonlinear phenomena, 195 (2004), pp.~159--187.

\bibitem{Berneretal}
{\sc R.~Berner, T.~Gross, C.~Kuehn, J.~Kurths, and S.~Yanchuk}, {\em Adaptive
  dynamical networks}, arXiv:2304.05652,  (2023), pp.~1--.

\bibitem{blonder2012temporal}
{\sc B.~Blonder, T.~W. Wey, A.~Dornhaus, R.~James, and A.~Sih}, {\em Temporal
  dynamics and network analysis}, Methods in Ecology and Evolution, 3 (2012),
  pp.~958--972.

\bibitem{boccaletti2006synchronization}
{\sc S.~Boccaletti, D.-U. Hwang, M.~Chavez, A.~Amann, J.~Kurths, and L.~M.
  Pecora}, {\em Synchronization in dynamical networks: Evolution along
  commutative graphs}, Physical Review E, 74 (2006), p.~016102.

\bibitem{bressan2007introduction}
{\sc A.~Bressan and B.~Piccoli}, {\em Introduction to the mathematical theory
  of control}, vol.~1, American institute of mathematical sciences Springfield,
  2007.

\bibitem{caraballo2008synchronization}
{\sc T.~Caraballo, P.~E. Kloeden, and A.~Neuenkirch}, {\em Synchronization of
  systems with multiplicative noise}, Stochastics and Dynamics, 8 (2008),
  pp.~139--154.

\bibitem{cenk2021edges}
{\sc M.~Cenk~Eser, E.~S. Medeiros, M.~Riza, and A.~Zakharova}, {\em Edges of
  inter-layer synchronization in multilayer networks with time-switching
  links}, Chaos: An Interdisciplinary Journal of Nonlinear Science, 31 (2021),
  p.~103119.

\bibitem{chen2007synchronization}
{\sc M.~Chen}, {\em Synchronization in time-varying networks: a matrix measure
  approach}, Physical Review E, 76 (2007), p.~016104.

\bibitem{coppel2006dichotomies}
{\sc W.~A. Coppel}, {\em Dichotomies in stability theory}, vol.~629, Springer,
  2006.

\bibitem{dieci2002lyapunov2}
{\sc L.~Dieci and E.~S. Van~Vleck}, {\em Lyapunov and other spectra: a survey},
  Collected lectures on the preservation of stability under discretization
  (Fort Collins, CO, 2001), 197 (2002), p.~218.

\bibitem{dieci2002lyapunov}
\leavevmode\vrule height 2pt depth -1.6pt width 23pt, {\em Lyapunov spectral
  intervals: theory and computation}, SIAM Journal on Numerical Analysis, 40
  (2002), pp.~516--542.

\bibitem{dieci2007lyapunov}
\leavevmode\vrule height 2pt depth -1.6pt width 23pt, {\em Lyapunov and
  {S}acker--{S}ell spectral intervals}, Journal of dynamics and differential
  equations, 19 (2007), pp.~265--293.

\bibitem{fink1974almost}
{\sc A.~M. Fink}, {\em Almost periodic differential equations}, vol.~377, 1974.

\bibitem{froyland2013computing}
{\sc G.~Froyland, T.~H{\"u}ls, G.~P. Morriss, and T.~M. Watson}, {\em Computing
  covariant lyapunov vectors, {O}seledets vectors, and dichotomy projectors: A
  comparative numerical study}, Physica D: Nonlinear Phenomena, 247 (2013),
  pp.~18--39.

\bibitem{gallotti2015multilayer}
{\sc R.~Gallotti and M.~Barthelemy}, {\em The multilayer temporal network of
  public transport in {G}reat {B}ritain}, Scientific data, 2 (2015), pp.~1--8.

\bibitem{ghosh2022synchronized}
{\sc D.~Ghosh, M.~Frasca, A.~Rizzo, S.~Majhi, S.~Rakshit, K.~Alfaro-Bittner,
  and S.~Boccaletti}, {\em The synchronized dynamics of time-varying networks},
  Physics Reports, 949 (2022), pp.~1--63.

\bibitem{hale1969ordinary}
{\sc J.~K. Hale}, {\em Ordinary Differential Equations}, Wiley-Interscience,
  New York, 1969.

\bibitem{holme2015modern}
{\sc P.~Holme}, {\em Modern temporal network theory: a colloquium}, The
  European Physical Journal B, 88 (2015), pp.~1--30.

\bibitem{holme2012temporal}
{\sc P.~Holme and J.~Saram{\"a}ki}, {\em Temporal networks}, Physics reports,
  519 (2012), pp.~97--125.

\bibitem{PhysRevE.84.046202}
{\sc H.~Hong and S.~H. Strogatz}, {\em Conformists and contrarians in a
  {K}uramoto model with identical natural frequencies}, Phys. Rev. E, 84
  (2011), p.~046202.

\bibitem{jeter2015synchronization}
{\sc R.~Jeter and I.~Belykh}, {\em Synchronization in on-off stochastic
  networks: Windows of opportunity}, IEEE Transactions on Circuits and Systems
  I: Regular Papers, 62 (2015), pp.~1260--1269.

\bibitem{jeter2019dynamics}
{\sc R.~Jeter, M.~Porfiri, and I.~Belykh}, {\em Dynamics and control of
  stochastically switching networks: beyond fast switching}, Temporal Network
  Theory,  (2019), pp.~269--304.

\bibitem{KissMillerSimon}
{\sc I.~Kiss, J.~Miller, and P.~Simon}, {\em Mathematics of Epidemics on
  Networks: From Exact to Approximate Models}, Springer, 2017.

\bibitem{kloeden2020introduction}
{\sc P.~Kloeden and M.~Yang}, {\em An introduction to nonautonomous dynamical
  systems and their attractors}, vol.~21, World Scientific, 2020.

\bibitem{kuhn2010distributed}
{\sc F.~Kuhn, N.~Lynch, and R.~Oshman}, {\em Distributed computation in dynamic
  networks}, in Proceedings of the forty-second ACM symposium on Theory of
  computing, 2010, pp.~513--522.

\bibitem{liberzon1999basic}
{\sc D.~Liberzon and A.~S. Morse}, {\em Basic problems in stability and design
  of switched systems}, IEEE control systems magazine, 19 (1999), pp.~59--70.

\bibitem{Liggett}
{\sc T.~Liggett}, {\em Stochastic Interacting Systems: Contact, Voter and
  Exclusion Processes}, Springer, 2013.

\bibitem{liu2011synchronization}
{\sc B.~Liu, W.~Lu, and T.~Chen}, {\em Synchronization in complex networks with
  stochastically switching coupling structures}, IEEE Transactions on Automatic
  Control, 57 (2011), pp.~754--760.

\bibitem{longo2018topologies}
{\sc I.~P. Longo}, {\em Topologies of continuity for {C}arath{\'e}odory
  differential equations with applications in non-autonomous dynamics}, Phd
  Thesis, Universidad de Valladolid, 2018.

\bibitem{longo2017topologies}
{\sc I.~P. Longo, S.~Novo, and R.~Obaya}, {\em Topologies of ${L}^p_{loc}$ type
  for {C}arath{\'e}odory functions with applications in non-autonomous
  differential equations}, Journal of Differential Equations, 263 (2017),
  pp.~7187--7220.

\bibitem{longo2019topologies}
\leavevmode\vrule height 2pt depth -1.6pt width 23pt, {\em Topologies of
  continuity for {C}arathéodory delay differential equations with applications
  in non-autonomous dynamics}, Discrete and Continuous Dynamical Systems, 39
  (2019), pp.~5491--5520.

\bibitem{longo2019weak}
\leavevmode\vrule height 2pt depth -1.6pt width 23pt, {\em Weak topologies for
  {C}arath{\'e}odory differential equations: continuous dependence, exponential
  dichotomy and attractors}, Journal of Dynamics and Differential Equations, 31
  (2019), pp.~1617--1651.

\bibitem{longo2021monotone}
\leavevmode\vrule height 2pt depth -1.6pt width 23pt, {\em Monotone
  skew-product semiflows for {C}arath{\'e}odory differential equations and
  applications}, Journal of Dynamics and Differential Equations,  (2021),
  pp.~1--37.

\bibitem{longo2022topologies}
{\sc I.~P. Longo, R.~Obaya, and A.~M. Sanz}, {\em Topologies of continuity for
  {C}arathéodory parabolic {PDE}s from a dynamical perspective}, Discrete and
  Continuous Dynamical Systems, 43 (2023), pp.~2213--2240.

\bibitem{lu2005time}
{\sc J.~Lu and G.~Chen}, {\em A time-varying complex dynamical network model
  and its controlled synchronization criteria}, IEEE Transactions on Automatic
  Control, 50 (2005), pp.~841--846.

\bibitem{lu2008synchronization}
{\sc W.~Lu, F.~M. Atay, and J.~Jost}, {\em Synchronization of discrete-time
  dynamical networks with time-varying couplings}, SIAM Journal on Mathematical
  Analysis, 39 (2008), pp.~1231--1259.

\bibitem{MulasKuehnJost}
{\sc R.~Mulas, C.~Kuehn, and J.~Jost}, {\em Coupled dynamics on hypergraphs:
  master stability of steady states and synchronization}, Phys. Rev. E, 101
  (2020), p.~062313.

\bibitem{olech1960inegalite}
{\sc C.~Olech and Z.~Opial}, {\em Sur une in{\'e}galit{\'e}
  diff{\'e}rentielle}, in Annales Polonici Mathematici, vol.~3, 1960,
  pp.~247--254.

\bibitem{OthmerScriven}
{\sc H.~Othmer and L.~Scriven}, {\em Non-linear aspects of dynamic pattern in
  cellular networks}, J. Theor. Biol., 43 (1974), pp.~83--112.

\bibitem{PecoraCarroll}
{\sc L.~Pecora and T.~Carroll}, {\em Master stability functions for
  synchronized coupled systems}, Phys. Rev. Lett., 80 (1998), pp.~2109--2112.

\bibitem{pereira2011stability}
{\sc T.~Pereira}, {\em Stability of synchronized motion in complex networks},
  arXiv preprint arXiv:1112.2297,  (2011).

\bibitem{pereira2013towards}
{\sc T.~Pereira, J.~Eldering, M.~Rasmussen, and A.~Veneziani}, {\em Towards a
  general theory for coupling functions allowing persistent synchronisation},
  arXiv preprint arXiv:1304.7679,  (2013).

\bibitem{porfiri2011master}
{\sc M.~Porfiri}, {\em A master stability function for stochastically coupled
  chaotic maps}, Europhysics Letters, 96 (2011), p.~40014.

\bibitem{porfiri2008synchronization}
{\sc M.~Porfiri, D.~J. Stilwell, and E.~M. Bollt}, {\em Synchronization in
  random weighted directed networks}, IEEE Transactions on Circuits and Systems
  I: Regular Papers, 55 (2008), pp.~3170--3177.

\bibitem{porfiri2006random}
{\sc M.~Porfiri, D.~J. Stilwell, E.~M. Bollt, and J.~D. Skufca}, {\em Random
  talk: Random walk and synchronizability in a moving neighborhood network},
  Physica D: Nonlinear Phenomena, 224 (2006), pp.~102--113.

\bibitem{porter2016dynamical}
{\sc M.~A. Porter and J.~P. Gleeson}, {\em Dynamical systems on networks},
  Frontiers in Applied Dynamical Systems: Reviews and Tutorials, 4 (2016).

\bibitem{potzsche2004exponential}
{\sc C.~P{\"o}tzsche}, {\em Exponential dichotomies of linear dynamic equations
  on measure chains under slowly varying coefficients}, Journal of Mathematical
  Analysis and Applications, 289 (2004), pp.~317--335.

\bibitem{rakshit2020intralayer}
{\sc S.~Rakshit, B.~K. Bera, E.~M. Bollt, and D.~Ghosh}, {\em Intralayer
  synchronization in evolving multiplex hypernetworks: Analytical approach},
  SIAM Journal on Applied Dynamical Systems, 19 (2020), pp.~918--963.

\bibitem{rauch1978qualitative}
{\sc J.~Rauch and J.~A. Smoller}, {\em Qualitative theory of the
  {F}itz{H}ugh-{N}agumo equations},  (1978).

\bibitem{rocsoreanu2012fitzhugh}
{\sc C.~Rocsoreanu, A.~Georgescu, and N.~Giurgiteanu}, {\em The FitzHugh-Nagumo
  model: bifurcation and dynamics}, vol.~10, Springer Science \& Business
  Media, 2012.

\bibitem{sacker1978spectral}
{\sc R.~J. Sacker and G.~R. Sell}, {\em A spectral theory for linear
  differential systems}, Journal of Differential Equations, 27 (1978),
  pp.~320--358.

\bibitem{saramaki2015seconds}
{\sc J.~Saram{\"a}ki and E.~Moro}, {\em From seconds to months: an overview of
  multi-scale dynamics of mobile telephone calls}, The European Physical
  Journal B, 88 (2015), pp.~1--10.

\bibitem{SegelLevin}
{\sc L.~Segel and S.~Levin}, {\em Application of nonlinear stability theory to
  the study of the effects of diffusion on predator-prey interactions}, AIP
  Conf. Proceed., 27 (1976), p.~123.

\bibitem{ShiAltafiniBaras}
{\sc G.~Shi, C.~Altafini, and J.~Baras}, {\em Dynamics over signed networks},
  SIAM Rev., 61 (2019), pp.~229--257.

\bibitem{siegmund2002dichotomy}
{\sc S.~Siegmund}, {\em Dichotomy spectrum for nonautonomous differential
  equations}, Journal of Dynamics and Differential Equations, 14 (2002),
  pp.~243--258.

\bibitem{spanos2005dynamic}
{\sc D.~P. Spanos, R.~Olfati-Saber, and R.~M. Murray}, {\em Dynamic consensus
  on mobile networks}, in IFAC world congress, 2005, pp.~1--6.

\bibitem{stilwell2006sufficient}
{\sc D.~J. Stilwell, E.~M. Bollt, and D.~G. Roberson}, {\em Sufficient
  conditions for fast switching synchronization in time-varying network
  topologies}, SIAM Journal on Applied Dynamical Systems, 5 (2006),
  pp.~140--156.

\bibitem{stopczynski2014measuring}
{\sc A.~Stopczynski, V.~Sekara, P.~Sapiezynski, A.~Cuttone, M.~M. Madsen, J.~E.
  Larsen, and S.~Lehmann}, {\em Measuring large-scale social networks with high
  resolution}, PloS one, 9 (2014), p.~e95978.

\bibitem{strogatz2001exploring}
{\sc S.~H. Strogatz}, {\em Exploring complex networks}, nature, 410 (2001),
  pp.~268--276.

\bibitem{tahbaz2006consensus}
{\sc A.~Tahbaz-Salehi and A.~Jadbabaie}, {\em On consensus over random
  networks}, in 44th Annual Allerton Conference, Citeseer, 2006.

\bibitem{VicsekZafiris}
{\sc T.~Vicsek and A.~Zafiris}, {\em Collective motion}, Phys. Rep.,  (2012),
  pp.~71--140.

\bibitem{wen2015pinning}
{\sc G.~Wen, W.~Yu, G.~Hu, J.~Cao, and X.~Yu}, {\em Pinning synchronization of
  directed networks with switching topologies: A multiple lyapunov functions
  approach}, IEEE Transactions on Neural Networks and Learning Systems, 26
  (2015), pp.~3239--3250.

\bibitem{zhang2021designing}
{\sc Y.~Zhang and S.~H. Strogatz}, {\em Designing temporal networks that
  synchronize under resource constraints}, Nature communications, 12 (2021),
  p.~3273.

\bibitem{zhang2014human}
{\sc Y.-Q. Zhang, X.~Li, J.~Xu, and A.~V. Vasilakos}, {\em Human interactive
  patterns in temporal networks}, IEEE Transactions on Systems, Man, and
  Cybernetics: Systems, 45 (2014), pp.~214--222.

\bibitem{zhao2010synchronization}
{\sc J.~Zhao, D.~J. Hill, and T.~Liu}, {\em Synchronization of dynamical
  networks with nonidentical nodes: Criteria and control}, IEEE Transactions on
  Circuits and Systems I: Regular Papers, 58 (2010), pp.~584--594.

\end{thebibliography}

\end{document}